\documentclass[11pt, leqno]{amsart}
\usepackage{texdraw,multicol,rotating,color}

\usepackage{amsmath}
\usepackage{amssymb}
\usepackage{enumerate}
\usepackage[all]{xy}
\usepackage{verbatim}

\input{txdtools.tex}

\let\goth\mathfrak
\def\gg{\goth g}

\def\gh{\goth h}

\def\gl{\goth l}

\def\gb{\goth b}

\def\gn{\goth n}

\def\gq{\goth q}

\def\ggl{\goth{gl}}

\def\beq{\begin{equation*}}
\def\eeq{\end{equation*}}
\def\bea{\begin{aligned}}
\def\eea{\end{aligned}}
\def\bee{\begin{enumerate}}
\def\eee{\end{enumerate}}

\def\ds{\displaystyle}
\def\cplus{\hbox{$\supset${\raise1.05pt\hbox{\kern -0.55em
${\scriptscriptstyle +}$}}\ }}

\DeclareMathOperator{\Hom}{Hom} 
 
 \DeclareMathOperator{\wt}{\rm{wt}}
 \DeclareMathOperator{\ch}{\rm{ch}}

\DeclareMathOperator{\End}{End}
\DeclareMathOperator{\Id}{Id}

\DeclareMathOperator{\Cliff}{{\rm Cliff}} 
\DeclareMathOperator{\ad}{\rm{ad}}

\numberwithin{equation}{section}

\theoremstyle{plain}
\newtheorem{theorem}{Theorem}[section]
\newtheorem{proposition}[theorem]{Proposition}
\newtheorem{lemma}[theorem]{Lemma}
\newtheorem{corollary}[theorem]{Corollary}

\theoremstyle{definition}
\newtheorem{definition}[theorem]{Definition}

\newtheorem{example}[theorem]{Example}
\newtheorem*{remark}{Remark}

\theoremstyle{remark}

\newtheorem*{acknowledgement}{Acknowledgements}

\def\N{\mathbb N}
\def\Z{\mathbb Z}
\def\C{\mathbb C}
\def\F{\mathbb F}

\def\K{\mathbb K}

\def\zz{\zeta}
\def\ee{\eta}
\def\fqk{f_{\zz}q^h k_{\ee}}
\def\fz{f_{\zz}}
\def\qh{q^h}

\def\fqke{f_{\zz}q^h k_{\ee} \et}

\def\dd{\delta_{ij}}

\def\al{\alpha}
\def\be{\beta}
\def\ep{\epsilon}
\def\zz{\zeta}
\def\la{\lambda}
\def\ee{\eta}

\def\De{\Delta}

\def\ot{\otimes}
\def\op{\oplus}

\def\ra{\longrightarrow}
\def\map{\longmapsto}

\def\fqk{f_{\zz}q^h k_{\ee}}
\def\fz{f_{\zz}}
\def\qh{q^h}

\def\et{e_{\tau}}
\def\fqke{f_{\zz}q^h k_{\ee} \et}
\def\ta{\tau}

\def\gdiw{\Lambda^{+}_{\bar 0}}
\def\qdiw{\Lambda^{+}}
\allowdisplaybreaks

\def\Ao{\mathbf{A}_1}
\def\UAo{U_{\Ao}}
\def\Uq{U_q}
\def\Uqqn{U_q(\gg)}
\def\Uqgeq{\Uq^{\geq 0}}
\def\Cliffqla{\Cliff_q(\la)}
\def\CliffAola{\Cliff_{\Ao}(\la)}

\def\Eib{e_{\bar i}}
\def\Ei{e_i}

\def\oEj{\overline{e_j}}
\def\Fib{f_{\bar i}}
\def\Fi{f_i}

\def\klb{k_{\bar l}}

\def\VAo{V_{\Ao}}
\DeclareMathOperator{\rank}{rank}
\def\Jo{\mathbf{J}_1}

\def\qhh{(q^h;0)_q}
\def\oEi{\overline{e_i}}
\def\oEib{\overline{\Eib}}
\def\oFi{\overline{f_i}}
\def\oFib{\overline{\Fib}}
\def\oh{\overline h}

\def\oklb{\overline{\klb}}

\begin{document}
\title[Highest weight modules over the quantum queer superalgebra $U_q(\gq(n))$]
{Highest Weight Modules over \\ Quantum Queer Superalgebra
$U_q(\gq(n))$}

\author[Grantcharov, Jung, Kang, Kim]{Dimitar Grantcharov$^1$, Ji Hye Jung$^{2,3}$, Seok-Jin Kang$^{2}$ and Myungho Kim$^{2,3}$}

\address{Department of Mathematics \\
         University of Texas at Arlington \\ Arlington, TX 76021, USA}

         \email{grandim@uta.edu}

\address{Department of Mathematical Sciences
         and
         Research Institute of Mathematics \\
         Seoul National University \\ San 56-1 Sillim-dong, Gwanak-gu \\ Seoul 151-747, Korea}

         \email{jhjung@math.snu.ac.kr}

\address{Department of Mathematical Sciences
         and
         Research Institute of Mathematics \\
         Seoul National University \\ San 56-1 Sillim-dong, Gwanak-gu \\ Seoul 151-747, Korea}

         \email{sjkang@math.snu.ac.kr}

\address{Department of Mathematical Sciences
         and
         Research Institute of Mathematics \\
         Seoul National University \\ San 56-1 Sillim-dong, Gwanak-gu \\ Seoul 151-747, Korea}

         \email{mkim@math.snu.ac.kr}

\thanks{$^{1}$This research was supported by a UT Arlington REP Grant.}
\thanks{$^{2}$This research was supported by KRF Grant \# 2007-341-C00001.}
\thanks{$^{3}$This research was supported by BK21 Mathematical Sciences Division.}

\begin{abstract}
In this paper, we investigate the structure of highest weight
modules over the quantum queer superalgebra $U_q(\gq(n))$. The key
ingredients are the triangular decomposition of $U_q(\gq(n))$ and
the classification of finite dimensional irreducible modules over
quantum Clifford superalgebras. The main results we prove are the
classical limit theorem and the complete reducibility theorem for
$U_q(\gq(n))$-modules in the category ${\mathcal O}_{q}^{\geq 0}$.
\end{abstract}

\maketitle

\section*{Introduction}

Since its inception, the representation theory of Lie superalgebras
has been known to be much more complicated than the corresponding
theory of Lie algebras. One of the Lie superalgebra series attracts
special attention due to its resemblance of the Lie algebra $\gg \gl_n$ on the one hand and because of the unique properties of its
structure and representations on the other. This is the so-called
{\it queer} (or {\it strange}) Lie superalgebra $\gq (n)$ which
consists of all endomorphisms of $\C^{n|n}$ that commute with an odd
automorphism $P$ of $\C^{n|n}$ such that $P^2 = \Id$. The
queer nature of $\gq(n)$ is partly due to the nonabelian structure
of its Cartan subsuperalgebra $\gh$ having a nontrivial odd part
$\gh_{\bar{1}}$.  Another unique property of $\gq(n)$ is that,
although it has no invariant bilinear form, it admits an invariant
odd bilinear form. Because of the nonabelian structure of $\gh$, the
study of the highest weight modules of $\gq(n)$ requires some tools
in addition to the standard technique. For example, the highest
weight space ${\bold v}_{\lambda}$ of an irreducible highest weight
$\gq(n)$-module $V(\lambda)$ has a Clifford module structure. The
case when $V(\lambda)$ is a tensor module; i.e., a submodule of some
tensor power $V^{\otimes r}$ of the natural $\gq(n)$-module $V =
\C^{n|n}$, was treated first by Sergeev in 1984. In \cite{Se2}
Sergeev established several important results, among which are the
complete reducibility of $V^{\otimes r}$, a  character formula of
$V(\lambda)$, and an analog of the fundamental Schur-Weyl duality,
often referred as Sergeev duality.  The characters of all simple
finite-dimensional $\gq(n)$-modules have been found by Penkov and
Serganova in 1996 (see \cite{PS2} and \cite{PS3}) via an algorithm
using a supergeometric version of the Borel-Weil-Bott Theorem. In 2004 Brundan, \cite{B},
obtained the character formula of Penkov and Serganova using a different approach and
formulated a conjecture for the characters of the irreducible modules in the category $\mathcal O$.
Important results related to the simplicity of the highest weight $\gq(n)$-modules were obtained recently by Gorelik in \cite{G}.

In this paper we initiate the study of highest weight
representations of the quantum superalgebra $U_q(\gq (n))$.  The aim
of this paper is twofold. We want to study highest weight
$U_q(\gq(n))$-modules on the one hand, and to build the foundations
of the crystal bases theory for the tensor modules of $U_q(\gq (n))$
on the other. The latter problem will be treated in a future work.

A quantum deformation of the universal enveloping algebra of $\gq
(n)$ was constructed first by Olshanski in \cite{O}. Olshanski's
construction is a flat deformation of the universal enveloping
algebra $U(\gq (n))$ of $\gq (n)$ and is a quantum enveloping
superalgebra in the sense of Drinfeld (\cite{Dr}, \S 7). The idea
in \cite{O} is to apply a suitable modification of the procedure
used by Faddeev, Reshetikhin, and Takhtajan in \cite{RTF} -- using
an element $S$ in $\End (\C^{n|n})^{\otimes 2}$ that satisfies the
quantum Yang-Baxter equation. However, as pointed out by
Olshanski, the $r$-matrix $r \in \gq(n)^{\otimes 2}$ does not
satisfy the classical Yang-Baxter equation. Thus no quantum
analogue of $U(\gq (n))$ can be a quasi-triangular Hopf algebra.

In the present paper, based on the description of Olshanski, we give
a presentation of $U_q(\gq (n))$ in terms of generators and
relations so that the relations are quantum deformations of the
relations of $\gq(n)$ obtained in \cite{LS}. Using this
presentation, we find a natural triangular decomposition of $U_q(\gq
(n))$, and then introduce the notion of highest weight modules and
Weyl modules. Similarly to the case of $\gq(n)$, in order to study
highest weight modules, one has to describe the modules over the
quantum Clifford superalgebra ${\rm Cliff}_q (\lambda)$ for a weight $\lambda$ of $\gq (n)$.
 These modules, as we show in Section
\ref{section_clifford}, do not have the same structure as the ones
over the classical Clifford superalgebra $\Cliff (\lambda)$. For
example, the irreducible modules over ${\rm Cliff}_q (\lambda)$ are
parity invariant for much larger set  of weights $\lambda$, compared
with the irreducibles over $\Cliff (\lambda)$.

In the last two sections of the paper we focus on the category
${\mathcal O}_{q}^{\geq 0}$ of finite dimensional $U_q(\gq
(n))$-modules all whose weights are of the form $\lambda_1
\epsilon_1 +\cdots+ \lambda_n \epsilon_n$ ($\lambda_i \in
\Z_{\geq 0}$). One of our main results is a classical limit
theorem for the irreducible modules in $ {\mathcal O}_q^{\geq 0}$. Due to the structure of the quantum Clifford superalgebra, the
classical limit theorem is non-standard, as it is not true in
general that the classical limit $V^1$ of an irreducible highest
weight $U_q(\gq(n))$-module $V^q(\lambda)$ is $V(\lambda)$. In
fact, as we show in Section \ref{section_limit}, if $\lambda$ has
even number of nonzero coordinates $\lambda_1>...>\lambda_{2k}$,
then $\ch V^1 =2 \ch V(\lambda)$. The ``queer'' version of the
classical limit theorems are Theorem \ref{V1_is_Uqn-module} and
Theorem \ref{U1_is_Uqn-module}. With the aid of the classical
limit theorems we obtain another important result in the last
section: the category ${\mathcal O}_{q}^{\geq 0}$  is semisimple.

The organization of the paper is as follows. In Section
\ref{section_q_n} we recall some definitions and basic results
about $\gq (n)$. The realization of $U_q(\gq(n))$ and its
triangular decomposition is provided in Section
\ref{section_quantum_q}. Section \ref{section_clifford} is devoted
to the study of the quantum Clifford superalgebra and its modules.
In Section \ref{section_highest} we introduce the notion of
highest weight modules and Weyl modules. In particular, we show that every Weyl
module $W^q(\lambda)$ has a unique irreducible quotient
$V^q(\lambda)$. The classical limit theorem for the category
${\mathcal O}_{q}^{\geq 0}$ is proved in Section
\ref{section_limit} and the complete reducibility of
$U_q(\gq(n))$-modules in ${\mathcal O}_{q}^{\geq 0}$ is
established in the last section.

\vskip 5mm

%%%%%%%%%%%%%%%%%%%%%%%%%%%%%%%%%%%%%%%%%%%%%%%%%%%%%%%%%%%%%%%%%%%%%%%%%%%%%%%%%%%%

\section{The Lie superalgebra $\gq(n)$ and its representations} \label{section_q_n}
The ground field  in this section will be $\C$. By $\Z_{\geq 0}$ and
$\Z_{> 0}$ we denote the nonnegative integers and strictly positive
integers, respectively. We set $\Z_2 = \Z / 2\Z$. Every vector space
$V = V_{\bar{0}} \oplus V_{\bar{1}}$ over $\C$ is $\Z_{2}$-graded
with even part $V_{\bar{0}}$ and odd part $V_{\bar{1}}$. We will
write $\dim V = m | n$ if $\dim V_{\bar{0}} = m$ and $\dim
V_{\bar{1}} = n$. By $\Pi$ we denote the parity change functor;
i.e., $\Pi V$ is a vector space for which $\Pi V_{\bar{0}} =
V_{\bar{1}}$ and $\Pi V_{\bar{1}} = V_{\bar{0}}$. The direct sum of
$r$ copies of a vector space $V$ will be written as $V^{\oplus r}$.

The Lie subsuperalgebra $\gg = \gq(n)$ of $\ggl (n|n)$ is defined in
matrix form by
\begin{equation*}
\gg = \gq(n):= \left\{ \left( \begin{array}{cc} A& B \\ B & A
\end{array} \right)\; \Big| \; A,B \in \gg \gl_n    \right\}.
\end{equation*}
By definition,
a subsupealgebra $\gh =\gh_{\bar{0}} \oplus \gh_{\bar{1}}$ of $\gg$ is a {\it Cartan subsuperalgebra}, if
it is a self-normalizing nilpotent subsuperalgebra.  Every such $\gh$ has a nontrivial odd
part $\gh_{\bar{1}}$.  We fix $\gh$ to be the {\it standard} Cartan subsuperalgebra, namely the one for which $\gh_{\bar{0}}$ has a
basis $\{k_{1},...,k_{n} \}$ and $\gh_{\bar{1}}$ has a basis
$\{k_{\bar{1}},...,k_{\bar{n}} \}$, where $k_{i}:=\left(
\begin{array}{cc} E_{i,i}& 0 \\ 0 & E_{i,i}
\end{array} \right)$, $k_{\bar{i}}:=\left(
\begin{array}{cc} 0& E_{i,i} \\ E_{i,i} & 0
\end{array} \right)$ and  $E_{i,j}$ is the $n \times n$ matrix having
$1$ in the $(i,j)$ position and $0$ elsewhere. One should note that all Cartan subsuperalgebras of $\gg$ are conjugate to $\gh$.  Let  $\{ \epsilon_1,...,\epsilon_n\}$  be the basis of
$\gh_{\bar{0}}^*$ dual to $\{ k_1,...,k_n \}$. We denote $k_i - k_{i+1}$ by $h_i$ for $i=1, 2, \cdots, n-1$. The root system $\Delta = \Delta_{\bar{0}} \cup
\Delta_{\bar{1}} $ of $\gg$ has identical even and odd parts.
Namely, $\Delta_{\bar{0}} =  \Delta_{\bar{1}} = \{\epsilon_i -
\epsilon_j\; | \; 1 < i\neq j < n \}$. In particular, the root space
decomposition $\gg = \bigoplus_{\alpha \in \Delta} \gg_{\alpha}$ has
the property that $\gg_{\alpha}$ has dimension $1|1$ for every $\alpha
\in \Delta$. Set $\alpha_i:=\epsilon_{i} -
\epsilon_{i+1}$.
Let $Q=\bigoplus_{i=1}^{n-1} \Z \alpha_i$ be the root lattice and
$Q_{+}=\sum_{i=1}^{n-1} \mathbb{Z}_{\geq 0} \alpha_i$ be the
positive root lattice. The notation $Q_{-}=-Q_{+}$ will also be
used. There is a partial ordering on $\mathfrak{h}_{\bar{0}}^*$
defined by $\lambda \geq \mu$ if and only if $\lambda - \mu \in
Q_{+}$ for $\lambda, \mu \in \mathfrak{h}_{\bar{0}}^* $. The root
space $\gg_{\alpha_i}$ is spanned by $e_i:=\left(
\begin{array}{cc} E_{i,i+1}& 0 \\ 0 & E_{i,i+1}
\end{array} \right)$ and $e_{\bar{i}} := \left(
\begin{array}{cc}0 &  E_{i,i+1} \\  E_{i,i+1} & 0
\end{array} \right)$, while $\gg_{-\alpha_i}$ is spanned by
$f_i:=\left(
\begin{array}{cc} E_{i+1,i}& 0 \\ 0 & E_{i+1,i}
\end{array} \right)$ and $f_{\bar{i}} := \left(
\begin{array}{cc} 0 & E_{i+1,i} \\ E_{i+1,i} & 0
\end{array} \right)$. Let $P:=\bigoplus_{i=1}^n \Z \epsilon_i$ be the
weight lattice of $\gg$ and denote by $P^{\vee} := \bigoplus_{i=1}^n
\Z k_i $ the dual weight lattice.

Let $I:=\{1,2, \cdots , n-1 \}$ and  $J:=\{1,2, \cdots, n\}$.
\begin{proposition} \cite{LS}
  The Lie superalgebra $\gg$ is generated by the elements
  $e_i, e_{\bar{i}},f_i, f_{\bar{i}}$ $(i \in I)$, $\gh_{\bar{0}}$ and
  $k_{\bar{l}}$ $(l \in J)$ with the following defining relations:
\allowdisplaybreaks
\begin{align*}
 & [h,h']=0 \ \  \text{for} \   h, h^{\prime} \in \gh_{\bar{0}}, \\
 & [h,e_i]=\al_i(h)e_i, \  [h,e_{\bar{i}}]=\alpha_i(h)e_{\bar{i}}  \ \  \text{for} \   h \in \gh_{\bar{0}}, \  i \in I,  \\
 & [h,f_i]=-\al_i(h)f_i, \  [h,f_{\bar{i}}]=-\alpha_i(h)f_{\bar{i}} \ \  \text{for} \   h \in \gh_{\bar{0}}, \  i \in I, \\
 & [h,k_{\bar{l}}]=0 \ \ \text{for} \  h \in \gh_{\bar{0}}, \ l \in J,\\
\allowdisplaybreaks
 & [e_i,f_j]=\dd(k_i-k_{i+1}), \ [e_i,f_{\bar{j}}]=\dd(k_{\bar{i}}-k_{\overline{i+1}}) \ \  \text{for} \ i, j \in I, \\
 & [e_{\bar i},f_j]=\dd(k_{\bar i}-k_{\overline{i+1}}), \ [k_{\overline l},e_i]=\al_i(k_{l})e_{\overline i} \ \  \text{for} \ i, j \in I, \ l \in J, \\
 & [k_{\overline l},f_i]=- \al_i(k_l) f_{\overline i}, \   [e_{\bar i},f_{\bar{j}}]=\dd(k_{i}+k_{i+1}) \  \text{for} \ i, j \in I, \ l \in J,   \\
 & [k_{\bar l},e_{\bar{i}}]= \begin{cases}
  e_i \ \ \ \text {if} \ \ l=i, i+1 \\
  0   \ \ \ \text{otherwise}\end{cases}
  \ \ \text{for} \ \ i \in I, \ l \in J, \\
 & [k_{\bar l}, f_{\bar{i}}]= \begin{cases}
  f_i \  \ \ \text {if} \  \ l=i, i+1 \\
  0   \ \ \ \text{otherwise} \end{cases}
   \ \ \text{for} \ \ i \in I, \ l \in J,\\
 & [e_i,e_{\bar j}] =[e_{\bar i},e_{\bar j}] = [f_i,f_{\bar j}] =[f_{\bar i},f_{\bar j}] =0 \ \ \ \text{for} \ i,j \in I, \ |i-j| \neq 1,  \\
 & [e_i,e_j]=[f_i,f_j]=0 \ \ \ \text{for} \ i,j \in I, \ |i-j| > 1,  \\
 & [e_{i}, e_{i+1}] = [e_{\bar{i}}, e_{\overline{i+1}}], [e_{i}, e_{\overline{i+1}}] = [e_{\bar{i}}, e_{i+1}], \\
 & [f_{i+1}, f_{i}] = [f_{\overline{i+1}}, f_{\bar{i}}],  [f_{i+1}, f_{\bar{i}}]=   [f_{\overline{i+1}}, f_{i}], \\
 & [k_{\bar i}, k_{\bar j}]=\dd 2k_{i} \ \ \text{for}  \ i, j \in J,\\
 & [e_i,[e_{i},e_j]] = [e_{\bar i}, [e_i, e_j]]=0 \ \ \ \text{for} \ i,j \in I, \ |i-j|=1,\\
 & [f_i,[f_{i},f_j]]= [f_{\bar i}, [f_i, f_j]]=0 \ \ \ \text{for} \  i,j \in I, \ |i-j|=1.
\end{align*}
\end{proposition}

\begin{remark}
 We modified the relations given in \cite{LS}. More precisely, we replaced the relations
\begin{equation} \bea
&[e_{\bar i},[e_i,e_{\bar j}]] = 0 \ \ \ \text{for} \  i,j \in I, \ |i-j|=1,\\
&[f_{\bar i},[f_i,f_{\bar j}]]=0 \ \ \ \text{for} \  i,j \in I, \ |i-j|=1 \label{old relation in LS}
\eea \end{equation}
by
\begin{equation} \bea
 & [e_{i}, e_{i+1}] = [e_{\bar{i}}, e_{\overline{i+1}}], [e_{i}, e_{\overline{i+1}}] = [e_{\bar{i}}, e_{i+1}], \\
 & [f_{i+1}, f_{i}] = [f_{\overline{i+1}}, f_{\bar{i}}],  [f_{i+1}, f_{\bar{i}}]=   [f_{\overline{i+1}}, f_{i}]. \label{new relation in LS}
\eea \end{equation}

Since \eqref{old relation in LS} can be derived from \eqref{new relation in LS} (and other ones), we can easily see that these two presentations are equivalent.
\end{remark}

The universal enveloping algebra $U(\gg)$ is obtained from the
tensor algebra $T(\gg)$ by factoring out by the ideal generated by
the elements $[u,v]-u \ot v + (-1)^{\al \be} v \ot u$, where $ \al
, \be \in \Z_2, \ u \in \gg_{\al}, \ v \in \gg_{\be}$. Let $U^+$
(respectively, $U^0$ and $U^-$) be the subalgebra of $U(\gg)$
generated by the elements $e_{i}, e_{\bar i}$ ($i\in I$)
(respectively, by $k_i, k_{\bar i}$ ($i \in J$) and by $f_i,
f_{\bar i}$ ($i \in I$)). By the Poincar\'e-Birkhoff-Witt theorem,
the universal enveloping algebra has the triangular decomposition:
\begin{equation}\label{eq:tri decomp of U(q(n))}
   U(\gg) \cong U^- \ot U^0 \ot U^+.
\end{equation}

%\begin{proposition}\label{tri decom of U(q(n))}
%  $U(\gg) \cong U^- \ot U^0 \ot U^+$.
%\end{proposition}

A $\gg$-module $V$ is called a {\it weight module} if it admits a {\it weight space decomposition}
   \begin{equation*}
     \begin{aligned}
       V = \bigoplus_{\mu \in \gh_{\bar 0}^*} V_{\mu},
       ~~\text{where}~~ V_{\mu} =\{ v \in V ~|~ h v = \mu(h)v ~~\text{for all}~~ h \in \gh_{\bar 0} \}.
     \end{aligned}
   \end{equation*}
For a weight $\gg$-module $M$ denote by $\wt (M)$ the set of weights $\lambda \in \gh_{\bar 0}^*$ for which $M_{\lambda} \neq 0$. Every submodule of a weight module is also a weight module. If
$\dim_{\mathbb{C}} V_{\mu} < \infty$ for all $\mu \in \gh_{\bar
0}^*$, the {\it character} of $V$ is defined to be
    \begin{equation*}
      \begin{aligned}
        \ch V = \sum_{\mu \in \gh_{\bar 0}^*} (\dim_{\mathbb{C}} V_{\mu})~e^{\mu},
      \end{aligned}
    \end{equation*}
where $e^{\mu}$ are formal basis elements of the group algebra
$\mathbb{C}[\gh_{\bar 0}^*]$ with the multiplication given by
$e^{\lambda}e^{\mu}=e^{\lambda + \mu}$ for all $\lambda, \mu \in
\gh_{\bar 0}^*$.

 Denote by $\gb_+$ the {\it
standard Borel subsuperalgebra} of $\gg$ generated by $k_l, k_{\bar{l}}$ ($l \in J$) and  $e_i$, $e_{\bar{i}}$ ($ i \in I$).  A weight module $V$ is called a {\it highest weight module} if it is
generated over $\gg$ by a finite dimensional irreducible
$\gb_+$-submodule (see \cite[Definition 4]{PS1}).
\begin{proposition} \cite{P}
Let $\bold v$ be a finite dimensional irreducible
$\mathbb{Z}_2$-graded $\gb_+$-module.
\begin{enumerate}%[{\quad \rm{()}}]
\item[{\rm (1)}]  The maximal nilpotent subsuperalgebra $\gn$ of $\gb_+$ acts on $\bold v$ trivially.
\item[{\rm (2)}]  For any weight $\mu \in \gh_{\bar 0}^*$, consider
the symmetric bilinear form $F_{\mu}(u,v):=\mu([u,v])$ on $\gh_{\bar
1}$ and let $\Cliff (\mu)$ be the Clifford superalgebra of the
quadratic space $(\gh_{\bar 1}, F_{\mu})$. Then
there exists a unique weight $\la \in \gh_{\bar 0}^*$ such that
${\bold v}$ is endowed with a canonical $\mathbb{Z}_2$-graded
$\Cliff (\la)$-module structure and ${\bold v}$ is determined by
$\lambda$ up to $\Pi$.
\item[{\rm (3)}]  $\gh_{\bar 0}$ acts on
${\bold v}$ by the weight $\la$ determined in {\rm (2)}.
\end{enumerate}
\end{proposition}
From the above proposition, we know that the dimension of the
highest weight space of a highest weight $\gg$-module with highest
weight $\la$ is the same as the dimension of an irreducible $\Cliff
(\la)$-module. On the other hand all irreducible $\Cliff(\la)$-modules have the
same dimension (see, for example, \cite[Table 2]{ABS}). Thus the dimension of the
highest weight space is constant for all highest weight modules with
highest weight $\la$.

\begin{definition} \label{Weyl module}
Let ${\bold v}(\la)$ be the irreducible $\gb_+$-module determined by
$\la$ up to $\Pi$. The {\it Weyl module} $W(\la)$ of $\gg$ with
highest weight $\la$ is defined to be
\begin{equation*}
W(\la):=U(\gg) \otimes _{U(\gb_+)} {\bold v(\la)}.
\end{equation*}
\end{definition}
\noindent Note that the structure of $W(\la)$ is determined by $\la$
up to $\Pi$.

\begin{remark}
One may define the {\it Verma module} corresponding to $\la$ by
$M(\la):=U(\gg) \otimes _{U(\gb_+)} {\Cliff (\la)}$. Since the Verma modules are not
highest weight modules, they will not be considered in this paper.
\end{remark}

We will denote by $\gdiw$ and $\qdiw$ the set of $\ggl_n$-{\it
dominant integral weights} and the set of $\gg$-{\it dominant
integral weights}, respectively. These are given by
\begin{align*}
&\gdiw := \{ \la_1\epsilon_1+\cdots+\la_n \epsilon_n\in
\gh_{\bar 0}^* ~|~ \la_i-\la_{i+1} \in \Z_{\geq 0} \  {\rm for \ all} \ i \in I \} \\
&\qdiw := \{ \la_1\epsilon_1+\cdots+\la_n \epsilon_n\in \gdiw ~|~
\la_i=\la_{i+1} \Rightarrow \la_i=\la_{i+1}=0 \ {\rm for \ all} \ i
\in I\}.
\end{align*}
\begin{comment}
A weight $\lambda=\la_1\epsilon_1+\cdots+\la_n \epsilon_n$ is
called {\it typical} if $\la_i+\la_j \neq 0$ for all $i,j$. The set
of typical dominant integral weights is denoted by $\Lambda^{++}$.
\end{comment}

% We have the following proposition for finite dimensional irreducible $\gg$-modules.

\begin{proposition} \cite{P} \label{f.d. irred g-mod}
\begin{enumerate}
\item[{\rm (1)}] For any weight $\la$, $W(\la)$ has a unique maximal submodule
$N(\la)$.

\item[{\rm (2)}] For each finite dimensional irreducible
$\gg$-module $V$, there exists a unique weight $\la \in \gdiw $ such
that $V$ is a homomorphic image of $W(\la)$.

\item[{\rm (3)}] $V(\la):=W(\la) / N(\la)$ is finite dimensional if and only if $\la \in
\qdiw$.
\begin{comment}
\item[{\rm (4)}] Let $W$ be the Weyl group of the Lie algebra
$\ggl(n)=\gg_{\bar 0}$, and let $s(w)=1$ (respectively, $-1$) if the
number of transpositions occurring in the element $w \in W$ is even
(respectively, odd).
    If $\la \in \Lambda^{++}$, then
 \beq
 \ch V = (1+\ep){\ds 2^{[(n-1)/2]}} \prod_{ \al \in \De_{\bar{1}}^+}
 \big(e^{ \al/2}+\ep e ^{-\al/2}\big) \Big\{\Big( \sum_{w \in W} s(w)e^{w(\la)}\Big)
 \Big/ \Big(\sum_{w \in W}s(w)e^{w(\rho_{0})}\Big)  \Big\}.
 \eeq
 \end{comment}
\end{enumerate}
\end{proposition}

Now we restrict our attention to the following subcategory of the category of
finite dimensional $\gg$-modules.
\begin{definition} \label{def_o}
 Set $P_{\geq 0}:=\{ \la=\lambda_1 \epsilon_1 +...+ \lambda_n
\epsilon_n \in P \ | \ \la_j \geq 0 \ \text{for all } j=1, \ldots,
n \}$. The {\it category ${\mathcal O}^{\geq 0}$} consists of
finite dimensional $U(\gg)$-modules $M$
with  weight space
decomposition $M=\bigoplus_{\mu \in P} M_{\mu }$
 satisfying $(i)$ $\wt(M)
\subset P_{\geq 0} $,
$(ii)$ $k_{\bar i}|_{M_{\mu}} = 0 $ for $\mu \in P_{\geq 0}$ and $i \in \{1, \ldots, n\}$
such that $\langle k_i, \mu \rangle =0$.
\end{definition}

\begin{remark} \label{new_remark}
The reason we impose condition $(ii)$ in Definition \ref{def_o} is that we want the category ${\mathcal O}^{\geq 0}$ and its $q$-analog ${\mathcal O}_q^{\geq 0}$ (see Definition \ref{def_o_q}) to be completely reducible. If $M$ is a weight module, then one can show that the ${\Cliff (\mu)}$-module $M_{\mu}$ is completely reducible if and only if $(ii)$ is satisfied. This follows from the complete reducibility criterion for ${\Cliff (\mu)}$-modules. For the $q$-version of this criterion see Lemma \ref{lambda_zero} and  Corollary \ref{cliff_mod}. In view of the above, one can verify that Definition \ref{def_o} is equivalent to the one where $(ii)$ is replaced by: $(ii)'$ $M$ is completely reducible as an ${\mathfrak h}$-module. The same applies to Definition \ref{def_o_q} where $(ii)$ can be replaced by the condition that $M$ is a completely reducible $U_q^0$-module.
\end{remark}

 One easily checks that ${\mathcal O}^{\geq 0}$ is closed under finite direct sum, tensor
product and taking submodules and quotient modules. As Sergeev showed, \cite{Se2}, for any $\lambda \in \qdiw \cap P_{\geq 0}$,  $V(\lambda)$ is a submodule of $V^{\otimes |\lambda|}$ where $V$ is the natural representation of $\gg$ and $|(\lambda_1,...,\lambda_n)| = \lambda_1 +...+ \lambda_n$.  This, together with the properties of the $\ggl_n$-module $V(\lambda)$ (see, for example, \cite[Theorem 7.2.3]{HK2002}), implies the following proposition.

\begin{proposition} \label{prop_o}
For each $\lambda \in \qdiw \cap P_{\geq 0}$, $V(\lambda)$ is an
irreducible $U(\gg)$-module in the category ${\mathcal
O}^{\geq 0}$. Conversely,
every irreducible $U(\gg)$-module in the category ${\mathcal
O}^{\geq 0}$ has the form $V(\lambda)$ for some $\lambda \in \qdiw
\cap P_{\geq 0}$.
\end{proposition}

In \cite{Se1}, Sergeev has presented an explicit set of generators of
$Z=\mathcal{Z}(U(\gg))$, the center of $U(\gg)$, and showed that
each Weyl module $W(\lambda)$ $(\lambda \in \gh_{\bar 0}^*)$ admits a
central character. Let $\chi_{\lambda} \in \Hom_{\C} (Z, \C)$ be the central character
afforded by $W(\lambda)$; i.e., every element $z \in Z$ acts on
$W(\lambda)$ as scalar multiplication by $\chi_{\lambda}(z)$. Following
\cite[(2.12)]{B}, to each weight $\la=\la_1\epsilon_1+\cdots+\la_n
\epsilon_n \in P$, one can assign a formal symbol
$$ \delta (\lambda) := \delta_{\lambda_1} + \cdots + \delta_{\lambda_n}$$
such that $\delta_0=0$ and $\delta_{-i}=-\delta_{i}$.

\begin{proposition}\cite[Theorem 4.19]{B}, \cite[Proposition 1.1]{PS2} \label{central character}
For $\lambda, \mu \in P$, $\chi_{\lambda}=\chi_{\mu}$ if and only if
$\delta (\lambda)=\delta (\mu)$.
\end{proposition}

The following proposition will be very useful in Section \ref{section_limit}.

\begin{proposition}\label{f.d. h.w. module in O is irred.}
Let $V$ be a finite dimensional highest weight module over $\gg$
with highest weight $\lambda \in \Lambda^+ \cap P_{\geq 0}$. Then
$V$ is isomorphic to an irreducible highest weight module
$V(\lambda)$.
\end{proposition}
\begin{proof}
If $V$ is reducible, since it is finite dimensional, it contains a
nonzero proper irreducible submodule $W$. Then $W$ is isomorphic to
an irreducible highest weight module $V(\mu)$ for some weight $\mu
\in \Lambda^+ \cap P_{\geq 0}$ by Proposition \ref{f.d. irred
g-mod}. We know that $\mu \lvertneqq \lambda$ and $\chi_{\lambda} =
\chi_{\mu}$. But, by Proposition \ref{central character}, $\delta(\lambda)=\delta (\mu)$.
Since $\lambda,\mu \in \Lambda^+ \cap P_{\geq 0}$, we have $\lambda=\mu$, which is a contradiction.
%It contradicts to $\mu \lvertneqq \lambda$.
Thus $V$ is irreducible and by Proposition \ref{f.d. irred g-mod},
it must be isomorphic to the irreducible highest weight module
$V(\lambda)$ up to $\Pi$.
\end{proof}

The next proposition gives a sufficient condition for the finite
dimensionality of a highest weight $\gg$-module.

\begin{proposition} \label{f.d. h.w. g-mod}
Let $V$ be a highest weight module over $\gg$ with highest weight
$\lambda \in \Lambda^{+}$. If $f_i^{\lambda(h_i)+1} v = 0$ for all
$v \in V_{\lambda}$ and $i \in I$, then $V$ is finite dimensional.
\end{proposition}
\begin{proof}
Let $\{x_1, x_2, \ldots, x_r\}$ and $\{y_1, y_2, \ldots, y_r\}$ be
bases of $\gg_{\bar 0}$ and $\gg_{\bar 1}$, respectively. Then by
the Poincar\'e-Birkhoff-Witt theorem, $U(\gg)$ has a basis
consisting of elements of the form $y_{1}^{\epsilon_1}
y_{2}^{\epsilon_2} \cdots y_{r}^{\epsilon_r} x_{1}^{n_1} x_{2}^{n_2}
\cdots x_{r}^{n_r}$ where $\epsilon_j = 0$ or $1$ and $n_j \in \N
\cup \{0\}$. Because $\{y_{1}^{\epsilon_1} y_{2}^{\epsilon_2} \cdots
y_{r}^{\epsilon_r} ~|~ \epsilon_j = 0, 1 \}$ is a finite set, it is
enough to show that $U(\gg_{\bar 0})V_{\lambda}$ is finite
dimensional. For any $v \in V_{\lambda}$, we know that $U(\gg_{\bar
0})v$ is a highest weight module over $\gg_{\bar 0}$ with highest
weight $\lambda$ satisfying $f_i^{\lambda(h_i)+1} v = 0$ for all
$i\in I$. Thus it is finite dimensional. Since $U(\gg_{\bar
0})V_{\lambda} \subset \displaystyle \sum_{v \in V_{\lambda}}
U(\gg_{\bar 0})v$, we have the desired result.
\end{proof}

We say that a weight $\la = \la_1\epsilon_1+\cdots+\la_n
\epsilon_n\in \gh_{\bar 0}^*$ is {\it $\alpha$-typical} if
$\alpha=\epsilon_i-\epsilon_j$ and $\la_i + \la_j \neq 0$. In
\cite{Se2}, Sergeev proved the following character formula for $V(\la)~
(\la \in \Lambda^+ \cap P_{\geq 0})$:
\begin{equation} \label{character_formula}
\ch V(\la)= \dfrac{\dim {\bold v_{\la}}}{D}  \sum_{w \in W}
{\rm sgn}~ w \ w \Big( e^{\lambda+\rho_0}
\prod_{\substack{\alpha \in \Delta_{{\bar 0}} ^+,\\ \la  \ {\rm is}
\ \alpha-{\rm tyipical}}} (1+e^{-\alpha})\Big),
\end{equation}
where ${\bold v_{\la}}$ is an irreducible $\Cliff(\la)$-module, $W$
is the Weyl group of $\gg_{{\bar 0}}=\ggl_n$, $\rho_0 =\frac{1}{2}
 \sum_{\alpha \in \Delta^+_{{\bar 0}}} \al$ and $D=\sum_{w \in
W} {\rm sgn} ~ w \ e^{w(\rho_0)}$ is the Weyl denominator. In
\cite{PS2}, the formula (\ref{character_formula}) is called the {\it
generic character formula} and an explicit algorithm for computing the character of an arbitrary finite
dimensional irreducible $\gg$-module is presented.
\vskip 5mm

%%%%%%%%%%%%%%%%%%%%%%%%%%%%%%%%%%%%%%%%%%%%%%%%%%%%%%%%%%%%%%%%%%%%%%%%%%%%%%%%%%%%%%%%%%%%%%%%%%%%%%%%%%%%%%%%%%%%%%

\section{The quantum superalgebra $U_{q}(\gq(n))$} \label{section_quantum_q}

In \cite{O}, Olshanski constructed the quantum deformation
$U_q(\gq(n))$ of the universal enveloping algebra of $\gq(n)$. The
quantum superalgebra $U_q(\gq(n))$ is defined to be the
associative algebra over $\C(q)$ generated by $L_{ij}, \ i \leq j,$
with defining relations
  \begin{equation} \label{defining_rel_Uqqn_in_Olshanski}
  \begin{aligned}
  & L_{ii}L_{-i,-i}=L_{-i,-i}L_{ii}=1,\\
  & (-1)^{p(i,j)p(k,l)}q^{\varphi(j,l)}L_{ij}L_{kl} + \{ k \leq j < l \} \theta(i,j,k)(q-q^{-1})L_{il}L_{kj}    \\
   & \hskip 12mm  + \{ i \leq -l < j \leq -k \} \theta(-i,-j,k)(q-q^{-1})L_{i,-l}L_{k,-j} \\
   & = q^{\varphi(i,k)} L_{kl} L_{ij} + \{ k < i \leq l \} \theta(i,j,k)(q-q^{-1})L_{il}L_{kj}  \\
   & \hskip 12mm + \{-l \leq i < -k \leq j \} \theta(-i,-j,k)(q-q^{-1})L_{-i,l}L_{-k,j} ~ ,
  \end{aligned}
 \end{equation}
 where $\varphi(i, j) = \delta_{|i|,|j|} {\rm sgn}(j)$,
 $~\theta(i,j,k)={\rm sgn}({\rm sgn}(i)+{\rm sgn}(j)+{\rm
sgn}(k))$, $p(i,j)=\begin{cases} 0 \ \ \text{if} \ \ ij >0  \\
                             1 \ \  \text{if} \ \ ij<0 ,
 \end{cases}$  for any indices $i \leq
j,~k \leq l$ in $\{ \pm 1,~ \cdots ~ \pm n \}$ and the symbol $\{ \cdot \cdot \cdot \}$ (the dots
stand for some inequalities) is equal to 1 if all of these
inequalities are fulfilled and 0 otherwise.

Following \cite[Remark 7.3]{O}, we consider the set of generators of $U_{q}(\gg)=U_q(\gq(n))$
as follows:
\begin{equation} \label{generators_of_olshanski}
\begin{aligned}
& q^{k_i}:=L_{i,i}, \  \ q^{-k_i}:=L_{-i,-i} , \ \ e_{i} := -
\frac{1}{q-q^{-1}}L_{-i-1, -i}, \ \ f_{i} :=
\frac{1}{q-q^{-1}}L_{i, i+1}, \\
& e_{\bar{i}} := - \frac{1}{q-q^{-1}}L_{-i-1, i}, \ \ f_{\bar{i}} :=
- \frac{1}{q-q^{-1}}L_{-i, i+1}, \ \ k_{\bar{i}} := -
\frac{1}{q-q^{-1}}L_{-i, i}.
\end{aligned}
\end{equation}

Our first main result is the following presentation of $U_q(\gg)$.
\begin{theorem} \label{defining relations of Uqqn}
The {\it quantum superalgebra} $U_{q}(\gg)$ is isomorphic to the unital
associative algebra over $\C (q)$ generated by the
elements $e_i,f_i,e_{\bar{i}}, f_{\bar{i}}$ $(i = 1,...,n-1)$,
$k_{\bar{l}}$ $(l = 1,...,n)$, and $q^h$ $(h \in P^{\vee})$,
satisfying the following relations
\begin{align}  \label{defining relations of our Uq(q(n))}
\allowdisplaybreaks
&q^0=1, q^{h_1 + h_2}=q^{h_1}q^{h_2} \mbox{ for } h_1,h_2 \in
P^{\vee}, \nonumber \\
\nonumber &q^h e_i q^{-h}= q^{\alpha_i(h)}e_i, q^h f_i q^{-h}=
q^{-\alpha_i(h)}f_i \mbox{ for } h \in P^{\vee}\\
\nonumber &q^h k_{\bar{i}} q^{-h} = k_{\bar{i}}, q^h e_{\bar{i}}q^{-h} =
q^{\alpha_i(h)}e_{\bar{i}}, q^h f_{\bar{i}}q^{-h} =
q^{-\alpha_i(h)}f_{\bar{i}} \mbox{ for } h \in P^{\vee}\\
\nonumber &e_if_i - f_i e_i = \frac{1}{q-q^{-1}}\left( q^{k_{i} - k_{i+1}} -
q^{-k_i+k_{i+1}} \right),\\
\nonumber & qe_{i+1}f_i - f_i e_{i+1} = e_i f_{i+1} - q f_{i+1}e_i = e_if_j - f_je_i = 0 \; \mbox{ if } |i-j|>1,\\
\nonumber & e_if_{\bar{i}} - f_{\bar{i}} e_i = q^{-k_{i+1}}k_{\bar{i}} - k_{\overline{i+1}}q^{-k_{i}},\\
\nonumber & qe_{i+1}f_{\bar{i}} - f_{\bar{i}} e_{i+1} = e_i f_{\overline{i+1}} - q f_{\overline{i+1}}e_i = e_if_{\bar{j}} - f_{\bar{j}}e_i = 0 \; \mbox{ if } |i-j|>1,\\
\nonumber &e_{\bar{i}}f_{i} - f_{i} e_{\bar{i}} = q^{k_{i+1}}k_{\bar{i}} - k_{\overline{i+1}}q^{k_{i}},\\
\nonumber &qe_{\overline{i+1}}f_{i} - f_{i} e_{\overline{i+1}} = e_{\bar{i}}
f_{i+1} - q f_{i+1}e_{\bar{i}} = e_{\bar{i}}f_{j} - f_{j}e_{\bar{i}} = 0 \; \mbox{ if } |i-j|>1,\\
\nonumber &k_{\bar{i}}e_i - q e_i k_{\bar{i}}= e_{\bar{i}}q^{-k_i}, \;
qk_{\bar{i}}e_{i-1} - e_{i-1}k_{\bar{i}}= - q^{-k_i}
e_{\overline{i-1}},\\
\nonumber &k_{\bar{i}}e_j - e_j k_{\bar{i}} = 0 \; \mbox{ for } j\neq i \mbox{ and } j \neq i-1,\\
\nonumber &k_{\bar{i}}f_i - q f_i k_{\bar{i}}= - f_{\bar{i}}q^{k_i}, \;
qk_{\bar{i}}f_{i-1} - f_{i-1}k_{\bar{i}}=  q^{k_i}
f_{\overline{i-1}},\\
\nonumber &k_{\bar{i}}f_j - f_j k_{\bar{i}}=0 \; \mbox{ for } j\neq i \mbox{
and
} j \neq  i-1, \\
&k_{\bar{i}}^2 = \frac{q^{2k_i} - q^{-2k_i}}{q^2-q^{-2}}, \;
k_{\bar{i}}k_{\bar{j}} = -k_{\bar{j}}k_{\bar{i}} \mbox{ for }
i\neq j,\\
\nonumber &e_{\bar{i}}f_{\bar{i}}+ f_{\bar{i}}e_{\bar{i}} = \frac{q^{k_i+k_{i+1}} - q^{-k_i-k_{i+1}}}{q-q^{-1}} +(q-q^{-1})k_{\bar{i}}k_{\overline{i+1}},\\
\nonumber &q e_{\overline{i+1}}f_{\bar{i}}+ f_{\bar{i}}e_{\overline{i+1}} =
e_{\bar{i}}f_{\overline{i+1}}+ q
f_{\overline{i+1}}e_{\bar{i}} = e_{\bar{i}}f_{\bar{j}}+ f_{\bar{j}}e_{\bar{i}} = 0 \; \mbox{ if } |i-j|>1,\\
\nonumber &k_{\bar{i}}e_{\bar{i}} + q e_{\bar{i}} k_{\bar{i}}=
e_{i}q^{-k_i}, \; qk_{\bar{i}}e_{\overline{i-1}} +
e_{\overline{i-1}}k_{\bar{i}}=
q^{-k_i} e_{i-1},\\
\nonumber &k_{\bar{i}}e_{\bar{j}} + e_{\bar{j}} k_{\bar{i}}=0 \; \mbox{ for
}
j\neq i \mbox{ and } j \neq i-1,\\
\nonumber &k_{\bar{i}}f_{\bar{i}} + q f_{\bar{i}} k_{\bar{i}}= f_{i}q^{k_i},
\; qk_{\bar{i}}f_{\overline{i-1}} + f_{\overline{i-1}}k_{\bar{i}}=
q^{k_i} f_{i-1},\\
\nonumber &k_{\bar{i}}f_{\bar{j}} + f_{\bar{j}} k_{\bar{i}}=0 \; \mbox{ for
}
j\neq i \mbox{ and } j \neq i-1,\\
\nonumber &e_{\bar{i}}^2 = - \frac{q-q^{-1}}{q+q^{-1}}e_i^2, \;
f_{\bar{i}}^2
= \frac{q-q^{-1}}{q+q^{-1}}f_i^2,\\
\nonumber &e_ie_j - e_je_i = f_i f_j - f_j f_i = e_{\bar{i}}e_{\bar{j}} +
e_{\bar{j}} e_{\bar{i}}  =
f_{\bar{i}}f_{\bar{j}} + f_{\bar{j}} f_{\bar{i}} = 0 \mbox{ if } |i-j| > 1,\\
\nonumber &e_{i}e_{\bar{j}} - e_{\bar{j}}e_{i}=f_{i}f_{\bar{j}} - f_{\bar{j}}f_{i}=0 \mbox{ if } |i-j| \neq 1,\\
\nonumber &e_i e_{i+1} - e_{i+1}e_i = e_{\bar{i}}e_{\overline{i+1}}+
e_{\overline{i+1}}e_{\bar{i}}, \;f_{i+1}f_i - f_i f_{i+1} =
f_{\bar{i}}f_{\overline{i+1}}+ f_{\overline{i+1}}f_{\bar{i}},\\
\nonumber &e_i e_{\overline{i+1}}- e_{\overline{i+1}}e_i  =
e_{\bar{i}}e_{i+1}-  e_{i+1}e_{\bar{i}}, \; f_{\overline{i+1}}f_i
- f_if_{\overline{i+1}} = f_{i+1}f_{\bar{i}} - f_{\bar{i}}f_{i+1}, \\
\nonumber & qe_i^2 e_{i+1} - (q+q^{-1})e_ie_{i+1}e_i + q^{-1}e_{i+1}e_i^2=0, \\
\nonumber & qf_i^2 f_{i+1} - (q+q^{-1})f_if_{i+1}f_i + q^{-1}f_{i+1}f_i^2=0,\\
\nonumber & qe_{i}e_{i+1}^2 - (q+q^{-1})e_{i+1}e_{i}e_{i+1} + q^{-1}e_{i+1}^2 e_{i}=0,  \\
\nonumber & qf_{i}f_{i+1}^2 - (q+q^{-1})f_{i+1}f_{i}f_{i+1} + q^{-1}f_{i+1}^2 f_{i}=0, \\
\nonumber & qe_i^2 e_{\overline{i+1}} - (q+q^{-1})e_ie_{\overline{i+1}}e_i +
q^{-1}e_{\overline{i+1}}e_i^2=0,  \\
\nonumber & qf_i^2 f_{\overline{i+1}} - (q+q^{-1})f_if_{\overline{i+1}}f_i + q^{-1}f_{\overline{i+1}}f_i^2=0, \\
\nonumber & qe_{\overline{i}}e_{i+1}^2 -  (q+q^{-1})e_{i+1} e_{\overline{i}}e_{i+1}
+ q^{-1}e_{i+1}^2 e_{\overline{i}}= 0, \\
\nonumber & qf_{\overline{i}}f_{i+1}^2 - (q+q^{-1})f_{i+1} f_{\overline{i}}f_{i+1} +
q^{-1}f_{i+1}^2 f_{\overline{i}}=0.
\end{align}
\end{theorem}

\begin{proof}
  Let $U$ be the unital associative algebra over $\C (q)$ generated by the
elements $e_i,f_i,e_{\bar{i}}, f_{\bar{i}}$ $(i = 1,...,n-1)$,
$k_{\bar{l}}$ $(l = 1,...,n)$, and $q^h$ $(h \in P^{\vee})$ with
defining relations given in \eqref{defining relations of our
Uq(q(n))}. Using \eqref{defining_rel_Uqqn_in_Olshanski} and
\eqref{generators_of_olshanski}, the relations in \eqref{defining
relations of our Uq(q(n))} can be derived easily. Thus there is a
well-defined algebra homomorphism $\phi : U \ra U_q(\gg)$.

From the relation \eqref{defining_rel_Uqqn_in_Olshanski}, we obtain
\begin{equation} \bea \label{induction_of_Lij}
     &L_{i,i+j}  = ~ (q-q^{-1})q^{- \sum_{h=1}^{j-1} k_{i+h}} \prod_{h=1}^{j-1} \ad f_{i+h} (f_i),  \\
     &L_{-i,i+j}  = ~-(q-q^{-1}) q^{- \sum_{h=1}^{j-1} k_{i+h}} \prod_{h=1}^{j-1} \ad f_{i+h} (f_{\bar{i}}) , \\
     &L_{-i-j,~i}  = ~(-1)^j (q-q^{-1})q^{\sum_{h=1}^{j-1} k_{i+h}} \prod_{h=1}^{j-1} \ad e_{i+h} (e_{\bar{i}}), \\
     &L_{-i-j,-i}  = ~(-1)^j (q-q^{-1})q^{\sum_{h=1}^{j-1} k_{i+h}} \prod_{h=1}^{j-1} \ad e_{i+h} (e_i),
\eea \end{equation} where $\ad b_i (b_j) :=b_ib_j-b_jb_i $,
$\prod^{j}_{h=1} \ad b_{i+h} (b_i) := \ad b_{i+j} \cdots \ad b_{i+1}
(b_i)$ and $\prod^{0}_{h=1} \ad b_{i+h} (b_i) =b_i$ for $b_i
=e_i,e_{\bar{i}},f_i, f_{\bar{i}} $ ($i=1, \cdots, n-1, \ j>0$). It follows that the homomorphism $\phi$ must be surjective.

It remains to prove $\phi$ is injective. For this purpose, we will
show that the relations in \eqref{defining_rel_Uqqn_in_Olshanski}
can be derived from the ones in \eqref{defining relations of our
Uq(q(n))}. The proof of our assertion is quite lengthy and tedious.
But the basic idea is just the case-by-case check-up.

\def\La{\Lambda}

We define the sets \begin{align*}
&\Lambda = \{(i,j) \in \Z / \{ 0 \} \times \Z / \{ 0 \}  \ | \ -n \leq i  \leq j \leq n \} , \ && \Lambda_1 = \{(i,j) \in \Lambda \ | \ i >0, j>0 \ \text{and} \ i < j \},\\
 & \Lambda_2 = \{(i,j) \in \Lambda \ | \ i <0, j>0 \ \text{and} \ |i| < |j| \},\ \ && \Lambda_3 = \{(i,j) \in \Lambda \ | \ i <0, j>0 \ \text{and} \ |i| > |j| \},\\
  & \Lambda_4  = \{(i,j) \in \Lambda \ | \ i <0, j<0 \ \text{and} \ |i| > |j| \},\ && \Lambda_5 = \{(i,j) \in \Lambda \ | \ i <0, j>0 \ \text{and} \ |i|=|j| \}.
\end{align*}

For $((i,j),(k,l)) \in \Lambda \times \Lambda $, let $a=\min \{
|i|,|j| \}, \ b=\max\{ |i|,|j| \}, \ c=\min \{ |k|,|l| \}, \
d=\max\{ |k|,|l| \}$. We list all possible subsets of $\Lambda \times \Lambda$:
\begin{align*} \allowdisplaybreaks
C_1&=\{((i,j),(k,l)) \in \Lambda \times \Lambda \ | \  c < d < a < b \}, && C_2=\{((i,j),(k,l)) \in \Lambda \times \Lambda \ | \  c < d = a < b \},\\
C_3&=\{((i,j),(k,l)) \in \Lambda \times \Lambda \ | \  c < a < d < b \}, && C_4=\{((i,j),(k,l)) \in \Lambda \times \Lambda \ | \  c < a < d = b \},\\
C_5&=\{((i,j),(k,l)) \in \Lambda \times \Lambda \ | \  c < a < b < d \},
&& C_6=\{((i,j),(k,l)) \in \Lambda \times \Lambda \ | \  c = a < d < b \},\\
 C_7&=\{((i,j),(k,l)) \in \Lambda \times \Lambda \ | \  c = a < d = b \},
&& C_8=\{((i,j),(k,l)) \in \Lambda \times \Lambda \ | \  c = a < b < d \},\\
 C_9&=\{((i,j),(k,l)) \in \Lambda \times \Lambda \ | \  a < c < d < b \},
&& C_{10}=\{((i,j),(k,l)) \in \Lambda \times \Lambda \ | \  a < c < d = b \},\\
C_{11}&=\{((i,j),(k,l)) \in \Lambda \times \Lambda \ | \  a < c < b < d \},
&& C_{12}=\{((i,j),(k,l)) \in \Lambda \times \Lambda \ | \  a < b = c < d \},\\
C_{13}&=\{((i,j),(k,l)) \in \Lambda \times \Lambda \ | \  a < b < c < d \}, &&  D_{1}=\{((i,j),(k,l)) \in \Lambda_5 \times \Lambda \ | \  |i|<c < d \},\\
D_{2} & =\{((i,j),(k,l)) \in \Lambda_5 \times \Lambda \ | \  |i|=c < d \}, && D_{3}=\{((i,j),(k,l)) \in \Lambda_5 \times \Lambda \ | \  c < |i| < d \},\\
D_{4}&=\{((i,j),(k,l)) \in \Lambda_5 \times \Lambda \ | \  c < |i| = d \}, && D_{5}=\{((i,j),(k,l)) \in \Lambda_5 \times \Lambda \ | \  c < d < |i| \},\\
D_{6}&=\{((i,j),(k,l)) \in \Lambda \times \Lambda_5 \ | \  |k| < a < b \}, && D_{7}=\{((i,j),(k,l)) \in \Lambda \times \Lambda_5 \ | \  |k|=a < b \},\\
D_{8}&=\{((i,j),(k,l)) \in \Lambda \times \Lambda_5 \ | \  a < |k| < b \}, && D_{9}=\{((i,j),(k,l)) \in \Lambda \times \Lambda_5 \ | \  a < b = |k| \},\\
D_{10}&=\{((i,j),(k,l)) \in \Lambda \times \Lambda_5 \ | \  a < b <|k| \}.
\end{align*}

We consider all cases for $\La_s \times \La_t \cap C_i$ ($1 \leq
s,t \leq 4, \ 1 \leq i \leq 13$) and $\La_s \times \La_t \cap
D_{i}$ ($ s=5, 1 \leq t \leq 4 $ or $1 \leq s \le 4, t=5$ and  $1 \leq i \leq 10$). Since the remaining
cases can be checked similarly, we just prove:
    \begin{align}
     &L_{i,i}L_{k,l}L_{i,i}^{-1}=q^{\varphi(l,i)-\varphi(k,i)}L_{k,l} && \text{if} \ \ (k,l) \in \La_1 \cup \La_2, \label{first relation}\\
     &L_{i,j}L_{k,l}- L_{k,l}L_{i,j}=0 \ \ \ &&\text{if} \ \  ((i,j),(k,l)) \in \La_1 \times \La_1 \cap C_{1}, \label{second relation} \\
    &L_{i,j}L_{k,l}- L_{k,l}L_{i,j}=(q-q^{-1})L_{i,l}L_{k,j} \ \ \ \ &&\text{if} \ \  ((i,j),(k,l)) \in \La_1 \times \La_1 \cap C_{2}, \label{third relation} \\
     &(L_{i,j})^2=\dfrac{q-q^{-1}}{q+q^{-1}} (L_{-i,j})^2 \ \  &&\text{if} \ \ (i,j) \in
      \La_2 \label{fourth relation}.
    \end{align}

From \eqref{induction_of_Lij}, we obtain
  \begin{equation*} \bea \label{relations of Lij}
     &L_{i,j}  = \frac{L_{j-1,j-1}^{-1}}{q-q^{-1}} (L_{j-1,j}L_{i,j-1}-L_{i,j-1}L_{j-1,j}) && \text{if} \ \ (i,j) \in \La_1 \cup \La_2, \\
    &L_{i,j}  = \frac{L_{-i-1,-i-1}}{q-q^{-1}} (L_{i,i+1}L_{i+1,j}-L_{i+1,j}L_{i,i+1}) && \text{if} \ \ (i,j)  \in \La_3 \cup \La_4.
\eea \end{equation*}
\noindent To prove \eqref{first relation}, we
use induction on $l-k$:
\begin{align*}
L_{i,i}L_{k,l}L_{i,i}^{-1} &=  \frac{L_{l-1,l-1}^{-1}}{q-q^{-1}} L_{i,i}(L_{l-1,l}L_{k,l-1}-L_{k,l-1}L_{l-1,l}) L_{i,i}^{-1}\\
 &=q^{\varphi(l,i)-\varphi (l-1,i) + \varphi (l-1,i)-\varphi(k,i)} \frac{L_{l-1,l-1}^{-1}}{q-q^{-1}} \big( L_{l-1,l}L_{k,l-1}-L_{k,l-1}L_{l-1,l} \big)\\
 &=q^{\varphi(l,i)-\varphi (k,i)} L_{k,l}.
\end{align*}

From \eqref{defining relations of our Uq(q(n))}, we know that
$f_if_j-f_jf_i=0$ if $|i-j| >1$. By using induction on $j-i$ and
\eqref{first relation}, one can show that
$L_{i,j}L_{k,k+1}-L_{k,k+1}L_{i,j}=0$  when $((i,j),(k,k+1))\in
\La_1 \times \La_1 \cap C_1$. Similarly, one can prove
$L_{i,j}L_{k,l}-L_{k,l}L_{i,j}=0$ by induction on $l-k$. The proof
of \eqref{third relation} is analogous (we use
induction on $l-k$ and \eqref{first relation}, \eqref{second
relation}):
\begin{align*}
  L_{i,j}L_{k,l} &=\frac{L_{l-1,l-1}^{-1}}{q-q^{-1}} L_{i,j}(L_{l-1,l}L_{k,l-1}-L_{k,l-1}L_{l-1,l})\\
                 &=\frac{L_{l-1,l-1}^{-1}}{q-q^{-1}} \big( L_{l-1,l}L_{i,j}L_{k,l-1}+(q-q^{-1})L_{i,l}L_{l-1,j}L_{k,l-1} -L_{k,l-1}L_{i,j}L_{l-1,l} \big)\\
                 &=\frac{L_{l-1,l-1}^{-1}}{q-q^{-1}} \big( L_{l-1,l}L_{k,l-1}L_{i,j}+(q-q^{-1})L_{i,l}L_{l-1,j}L_{k,l-1} -L_{k,l-1}L_{l-1,l}L_{i,j}\\
                 & \hskip5em -(q-q^{-1})L_{k,l-1}L_{i,l}L_{l-1,j} \big)\\
                  &=L_{k,l}L_{i,j}+L_{l-1,l-1}^{-1} L_{i,l}(L_{l-1,j}L_{k,l-1}-L_{k,l-1}L_{l-1,j})\\
                 &=L_{k,l}L_{i,j}+(q-q^{-1})L_{i,l}L_{k,j}.
\end{align*}

 To verify the relation \eqref{fourth relation}, it suffices to show that
\begin{align*}
  (L_{j-1,j}L_{i,j-1}-L_{i,j-1}L_{j-1,j})^2=\dfrac{q-q^{-1}}{q+q^{-1}} (L_{j-1,j}L_{-i,j-1}-L_{-i,j-1}L_{j-1,j})^2.
\end{align*}
For this purpose, we need the following formulas for $(i,j) \in \La_2$ which can be
derived using induction:
\begin{align*}
  &L_{j-1,j}L_{i,j-1}L_{j-1,j}=\dfrac{1}{q+q^{-1}}(q L_{i,j-1}L_{j-1,j}^2+q^{-1}L_{j-1,j}^2L_{i,j-1}),\\
  &qL_{-i,j-1}L_{j-1,j}^2 -(q+q^{-1})L_{j-1,j}L_{-i,j-1}L_{j-1,j}+q^{-1}L_{j-1,j}^2L_{-i,j-1}=0.
\end{align*}

Using these formulae, we can verify the desired relations
\begin{align*}
  &(L_{j-1,j}L_{i,j-1}-L_{i,j-1}L_{j-1,j})^2 \\
  %&=L_{j-1,j}L_{i,j-1}L_{j-1,j}L_{i,j-1}-L_{j-1,j}L_{i,j-1}^2 L_{j-1,j} -L_{i,j-1}L_{j-1,j}^2L_{i,j-1}+L_{i,j-1}L_{j-1,j}L_{i,j-1}L_{j-1,j}\\
  &=(L_{j-1,j}L_{i,j-1}L_{j-1,j})L_{i,j-1}-\frac{q-q^{-1}}{q+q^{-1}}L_{j-1,j}L_{-i,j-1}^2 L_{j-1,j} -L_{i,j-1}L_{j-1,j}^2L_{i,j-1} \\
  & \hskip1em +L_{i,j-1}(L_{j-1,j}L_{i,j-1}L_{j-1,j})\\
  %&=\frac{1}{q+q^{-1}} \Big( q L_{i,j-1} L_{j-1,j}^2 L_{i,j-1} +q^{-1} L_{j-1,j}^2 L_{i,j-1}^2 +q L_{i,j-1}^2 L_{j-1,j}^2 +q^{-1} L_{i,j-1} L_{j-1,j}^2 L_{i,j-1} \\
  %&  \hskip5em  -(q+q^{-1})L_{i,j-1} L_{j-1,j}^2 L_{i,j-1} \Big) - \frac{q-q^{-1}}{q+q^{-1}}L_{j-1,j}L_{-i,j-1}^2 L_{j-1,j}       \\
  &= \frac{q-q^{-1}}{q+q^{-1}} \Big( \frac{q^{-1}}{q+q^{-1}}  L_{j-1,j}^2 L_{-i,j-1}^2 + \frac{q}{q+q^{-1}}  L_{-i,j-1}^2 L_{j-1,j}^2 -         L_{j-1,j}L_{-i,j-1}^2 L_{j-1,j} \Big) \\
  &= \frac{q-q^{-1}}{q+q^{-1}} \Big( ( L_{j-1,j} L_{-i,j-1} L_{j-1,j} - \frac{q}{q+q^{-1}}L_{-i,j-1} L_{j-1,j}^2 )L_{-i,j-1} \\
  &  +L_{-i,j-1}(L_{j-1,j} L_{-i,j-1} L_{j-1,j} - \frac{q^{-1}}{q+q^{-1}} L_{j-1,j}^2 L_{-i,j-1} ) -L_{j-1,j}L_{-i,j-1}^2 L_{j-1,j}                  \Big) \\
  %&= \frac{q-q^{-1}}{q+q^{-1}} \big( L_{j-1,j}L_{-i,j-1}L_{j-1,j}L_{-i,j-1}-L_{-i,j-1}L_{j-1,j}^2L_{-i,j-1}+L_{-i,j-1}L_{j-1,j}L_{-i,j-1}L_{j-1,j}  \\
  %&  \hskip5em  -L_{j-1,j}L_{-i,j-1}^2 L_{j-1,j} \big)\\
  &=\frac{q-q^{-1}}{q+q^{-1}}\big( L_{j-1,j}L_{-i,j-1}-L_{-i,j-1}L_{j-1,j} \big)^2.
 \end{align*} \end{proof}

Set $\deg f_i =\deg f_{\bar{i}}=-\alpha_i,~ \deg q^h =\deg
k_{\bar{l}}=0,~\deg e_i =\deg e_{\bar{i}}=\alpha_i $. Since all the
defining relations of the quantum superalgebra $U_q(\gg)$ are
homogeneous, it has a root space decomposition
\begin{equation*}
\begin{aligned}
  U_q(\gg) = \bigoplus_{\alpha \in Q} (U_q)_{\alpha},
\end{aligned}
\end{equation*}
where $(U_q)_{\alpha}~=~\{u \in U_q(\gg) ~| ~ q^h u q^{-h}=
q^{\alpha(h)}u ~~\text{for all}~ h \in P^{\vee}\}$.

\begin{remark}
 If we define
 \beq  F_i=f_i q^{-k_{i+1}}, \ E_i=q^{k_{i+1}}e_i,
 \eeq
one can see that the relations involving $E_i$, $F_i$ and $q^h$ are
the same as the standard relations for $U_q(\goth{gl}_n)$ (see, for
example, \cite[Definition 7.1.1]{HK2002}). Hence $U_q(\goth{gl}_n)$
is a subalgebra of $\Uqqn$.
\end{remark}

The comultiplication $\Delta$ of $U_q(\gg)$ is given
by the formula
\begin{equation}\label{comult}
\Delta(L_{i,j}) = \sum_{k=i}^j L_{i,k} \otimes L_{k,j},
\end{equation}
(see \S 4 in \cite{O}). In terms of the new generators we have:

\begin{align*}\label{Delta}
\Delta(q^h) & =  q^h \otimes q^h \mbox{ for every }h \in P^{\vee},\\
\Delta(e_i) & =  q^{-k_{i+1}} \otimes e_i + e_i \otimes q^{-k_i},\\
\Delta(f_i) & =  q^{k_{i}} \otimes f_i + f_i \otimes q^{k_{i+1}},\\
\Delta(e_{\bar{i}}) & =  q^{-k_{i+1}} \ot e_{\bar{i}} -(q-q^{-1})e_i \ot k_{\bar{i}} \\
     & +(q-q^{-1}) \Bigg(\sum_{j=1}^{i-1} (-1)^{j+1}   q^{\sum_{h=1}^{j} k_{i-j+h}} \prod_{h=1}^{j} \ad e_{i-j+h} (e_{i-j}) \\
     & \ot q^{- \sum_{h=1}^{j-1} k_{i-j+h}} \prod_{h=1}^{j-1} \ad f_{i-j+h} (f_{\overline{i-j}}) \Bigg)   \\
&  +(q-q^{-1}) \Bigg(\sum_{j=1}^{i-1} (-1)^{j}   q^{\sum_{h=1}^{j} k_{i-j+h}} \prod_{h=1}^{j} \ad e_{i-j+h} (e_{\overline{i-j}}) \\
& \ot q^{- \sum_{h=1}^{j-1} k_{i-j+h}} \prod_{h=1}^{j-1} \ad f_{i-j+h} (f_{i-j})  \Bigg) + e_{\bar{i}} \ot q^{k_{i}}, \\
\Delta(f_{\bar{i}}) & =  q^{-k_i} \ot f_{\bar{i}}  \\
&  + (q-q^{-1}) \Bigg(\sum_{j=1}^{i-1}  (-1)^j  q^{\sum_{h=1}^{j-1} k_{i-j+h}} \prod_{h=1}^{j-1} \ad e_{i-j+h} (e_{i-j}) \\
& \ot q^{- \sum_{h=1}^{j} k_{i-j+h}} \prod_{h=1}^{j} \ad f_{i-j+h} (f_{\overline{i-j}}) \Bigg)   \\
&  + (q-q^{-1}) \Bigg(\sum_{j=1}^{i-1} (-1)^{j+1}   q^{\sum_{h=1}^{j-1} k_{i-j+h}} \prod_{h=1}^{j-1} \ad e_{i-j+h} (e_{\overline{i-j}}) \\
& \ot q^{- \sum_{h=1}^{j} k_{i-j+h}} \prod_{h=1}^{j} \ad f_{i-j+h} (f_{i-j})  \Bigg)  \\
&   + (q-q^{-1}) ~ k_{\bar{i}} \ot f_i +f_{\bar{i}} \ot q^{k_{i+1}}, \\
\Delta(k_{\bar{i}}) & =  q^{-k_i} \ot k_{\bar{i}}  \\
&   + (q-q^{-1}) \Bigg(\sum_{j=1}^{i-1}  (-1)^j   q^{\sum_{h=1}^{j-1} k_{i-j+h}} \prod_{h=1}^{j-1} \ad e_{i-j+h} (e_{i-j}) \\
&  \ot q^{- \sum_{h=1}^{j-1} k_{i-j+h}} \prod_{h=1}^{j-1} \ad f_{i-j+h} (f_{\overline{i-j}})\Bigg)  \\
&  + (q-q^{-1})\Bigg(\sum_{j=1}^{i-1} (-1)^{j+1} q^{\sum_{h=1}^{j-1}
k_{i-j+h}} \prod_{h=1}^{j-1} \ad e_{i-j+h} (e_{\overline{i-j}}) \\
& \ot q^{- \sum_{h=1}^{j-1} k_{i-j+h}} \prod_{h=1}^{j-1} \ad f_{i-j+h} (f_{i-j}) \Bigg) +  k_{\bar{i}} \ot q^{k_{i}}.
\end{align*}

Let $U_q^{+}$ (respectively, $U_q^{-}$) be the subalgebra of
$U_q(\gg)$ generated by the elements $e_i, e_{\bar{i}}$
(respectively, $f_i, f_{\bar{i}}$) for $i=1, ... , n-1$, and let
$U_q^{0}$ be the subalgebra of  $U_q(\gg)$ generated by $q^h$ ($h
\in P^{\vee}$) and $k_{\bar{l}}$ for $l=1,... ,  n$. In addition,
let $U_q^{\geq 0}$ (respectively, $U_q^{\leq 0}$) be the subalgebra
of $U_q(\gg)$ generated by $U_q^{+}$ and $U_q^0$ (respectively, by
$U_q^{-}$ and $U_q^0$). We will show that the quantum
superalgebra $U_q(\gg)$ has a triangular decomposition. For this
purpose, we need the following lemma.

\begin{lemma} \label{half decom} \hfill
$$U_q^{\geq 0} \cong U_q^{0} \otimes U_q^{+}, \qquad
U_q^{\leq 0} \cong U_q^{-} \otimes U_q^{0}.$$
\end{lemma}

\begin{proof}
We will prove the second part. Let $\{f_{\zz} ~|~ \zz \in
\Omega \}$  be a basis of $U_q^{-}$ consisting of monomials in $f_i$
and $f_{\bar{i}}$'s  ($i \in I$). Consider a set $\Omega' = \{(a_1,\ldots,a_n) ~|~ a_i=0 \ \mbox{or} \ 1 \ \mbox{for all} \ i \in J \}$. Then   $\{q^h k_{\ee}
~|~ h \in P^{\vee},~ \ee \in \Omega ^{'} \}$ is a basis of
$U_q^{0}$, where $k_{\ee}=k_{\bar{1}}^{a_1} \cdots
k_{\bar{n}}^{a_n}$ for $\eta =(a_1, \cdots,a_n)$ by \cite[Theorem 6.2]{O}.
 By the defining relations of $U_q (\gg)$, it is easy to see that
the elements $f_{\zz}q^h k_{\ee} ~(\zz \in \Omega, ~h \in P^{\vee},
~\ee \in \Omega^{'})$ span $U_q^{\leq 0}$. Thus there is a
surjective $\mathbb{C}(q)$-linear map  $U_q^{-} \otimes U_q^{0}
\longrightarrow U_q^{\leq 0}$ given by $f_{\zz} \otimes q^hk_{\ee}
\longrightarrow f_{\zz}q^hk_{\ee}$. To show that this map is
injective, it suffices to show that the elements
$f_{\zz}q^hk_{\ee}~(\zz \in \Omega, ~h \in P^{\vee}, ~\ee \in
\Omega^{'})$ are linearly independent over $\mathbb{C}(q)$.

Suppose
\begin{equation*}
    \begin{aligned}
      \sum_{\substack{\zz \in \Omega ,~ h \in P^{\vee} , \\
      \ee \in \Omega^{'} }} C_{\zz, h, \ee} \fqk =0  \ \ \ \
      \text{for some }~~ C_{\zz, h, \ee} \in \mathbb{C}(q).
    \end{aligned}
  \end{equation*}
We may write
  \begin{equation*}
    \begin{aligned}
      \sum_{\beta \in Q_{+}}\Big(\sum_{\substack {\deg f_{\zz}=-\beta, \\ h \in P^{\vee},~ \ee \in \Omega^{'}}}
      C_{\zz, h, \ee} \fqk \Big) =0  \ \ \ \ \text{for some }~~ C_{\zz, h, \ee} \in \mathbb{C}(q).
    \end{aligned}
  \end{equation*}
Since $U_q(\gg) = \bigoplus_{\beta \in Q} (U_q)_{\beta}$, we have
\begin{equation*}
    \begin{aligned}
      \sum_{\substack {\deg f_{\zz}=-\beta,\\ h \in P^{\vee},~ \ee \in \Omega^{'}}}
      C_{\zz, h, \ee} \fqk =0  \ \ \ \ \text{for each }~~ \beta \in Q_{+}.
    \end{aligned}
    \end{equation*}

Write $\beta = -\sum_{i=1}^{n-1} m_i \alpha_i ~(m_i \in
\mathbb{Z}_{\geq 0})$, and let $h_{\beta}=\sum^{n-1}_{i=1} m_i
k_{i+1}$. Since $f_{\zz}$ is a monomial in $f_i$ and
$f_{\bar{i}}$'s, the term of degree $(-\beta,0)$ in $\Delta(\fz)$
is $\fz \otimes q^{h_{\beta}}$. We consider the terms of degree
$(0,0)$ in $\Delta(k_{\ee})$ where $\ee=(a_1,
\cdots, a_n)$. Then the terms of degree (0,0) in $\Delta(k_{\ee})$
can be written as
\begin{equation*}
  \begin{aligned}
    &(q^{-k_1}\otimes k_{\bar{1}}+k_{\bar{1}} \otimes q^{k_1})^{a_1} \cdots (q^{-k_n}\otimes k_{\bar{n}}+k_{\bar{n}} \otimes q^{k_n})^{a_n} \\
    &= \prod _{i=1}^{n}\Big( \sum^{a_i}_{j=0} q^{-(a_i -j)k_i}k_{\bar{i}}^j \ot q^{j k_i} k_{\bar{i}}^{a_i -j} \Big)\\
    &=\sum_{(j_1,\ldots,j_n) \in \Omega' \atop j_i \leq a_i, \ i \in J}  \prod_{i=1}^{n} \Big( q^{-(a_i -j_i)k_i}k_{\bar{i}}^{j_i} \ot q^{j_i k_i} k_{\bar{i}}^{a_i -j_i} \Big).
  \end{aligned}
\end{equation*}
Since the terms of degree $(-\beta,0)$ of $\sum C_{\zz, h, \ee}
\Delta(\fqk)$ must sum to zero, we have
\begin{equation} \label{terms of degree (-beta,0)}
0=\sum_{\ee=(a_1,\ldots,a_n) \atop \in \Omega'}
    \sum_{(j_1,\ldots,j_n) \in ~ \Omega' \atop j_i \leq a_i, \ i \in J} \Bigg(\sum_{\deg f_{\zz}
    =-\beta, \atop h \in P^{\vee}} C_{\zz, h, \ee} \fz \qh
    \Big( \prod^{n}_{i=1}
    q^{-(a_i -j_i)k_i}k_{\bar{i}}^{j_i} \Big) \ot q^{h_{\beta} +h}
    \Big( \prod_{i=1}^{n} q^{j_i k_i} k_{\bar{i}}^{a_i -j_i} \Big) \Bigg).
\end{equation}

For all $ (a_1-j_1, \cdots , a_n-j_n) \in \Omega^{'} $ and $h \in P^{\vee}$, the elements
$q^{h} \big(\prod_{i=1}^{n} k_{\bar{i}}^{a_i -j_i}\big)$ are
linearly independent. Set $\eta_1:=(1, \ldots,1)$. Since there is only one pair of $(a_1, \ldots, a_n)$ and $(j_1, \ldots, j_n)$ such that $\prod_{i=1}^n k_{\bar i}^{a_i-j_i} = k_{\eta_1}$ in the above sum, we obtain
\begin{align*}
0=&\sum_{h \in P^{\vee}} \sum_{\deg \fz=-\beta} C_{\zz, h, \ee_1} \fz q^{h-\sum_{i=1}^n k_i} \ot q^{h_{\beta} +h} k_{\ee_1}  \\
&+\sum_{\eta=(a_1, \ldots, a_n), ~ (j_1, \ldots,j_n),  \atop  (a_1-j_1,\ldots,a_n-j_n) \neq \ee_1} \sum_{h \in P^{\vee}, \atop \deg \fz=-\beta } C_{\zz, h, \ee} \fz \qh
    \Big( \prod^{n}_{i=1}
    q^{-(a_i -j_i)k_i}k_{\bar{i}}^{j_i} \Big) \ot q^{h_{\beta} +h}
    \Big( \prod_{i=1}^{n} q^{j_i k_i} k_{\bar{i}}^{a_i -j_i} \Big).
\end{align*}

Thus we have
\begin{align*}
    &\sum_{\deg f_{\zz}=-\beta} C_{\zz, h, \ee_1} \fz q^{h-\sum_{i=1}^n k_i}  =0 \ \ \text{for all} \ h \in P^{\vee}.
\end{align*}

Multiplying by $q^{-h+\sum_{i=1}^n k_i}$ from the right  we obtain
\begin{align*}
    &\sum_{\deg f_{\zz}=-\beta} C_{\zz, h, \ee_1} \fz  =0  \ \ \text{for all} \ h \in P^{\vee}.
\end{align*}
Using the linear independence of $\fz$, we conclude all $C_{\zeta,
h, \eta_1}=0$ for all $\zeta \in \Omega, h \in P^{\vee}$.
 Now consider general $\eta=(a_1,\ldots,a_n) \in \Omega'$. Assume that for all $\eta'=(a'_1, \ldots, a'_n)$ such that $a'_i \geq a_i$ for all $i \in J$ and $\eta' \neq \eta$, $C_{\zeta,h,\eta'}=0$ for all $\zeta \in \Omega, h \in P^{\vee}$. Then there is only one pair of $(a_1, \ldots, a_n)$ and $(j_1, \ldots, j_n)$ such that $(a_1-j_1, \ldots a_n-j_n)=\eta$ in \eqref{terms of degree (-beta,0)}. Repeating the above argument, we conclude $C_{\zz, h, \ee} =0
$ for all $\zeta \in \Omega, h \in P^{\vee}$.

For example, consider $\eta_2=(0,1,\ldots,1)$. Since $C_{\zeta,h, \eta_1}=0$, there is only one pair of $(a_1, \ldots, a_n)$ and $(j_1, \ldots, j_n)$ such that $(a_1-j_1, \ldots a_n-j_n)=(0,1,\ldots,1)$ in \eqref{terms of degree (-beta,0)}. Thus  we have
\begin{align*}
 &\sum_{\deg f_{\zz}=-\beta}   C_{\zz, h, \ee_2} \fz q^{h-\sum_{i=2}^n k_i} =0 \ \ \text{for all} \ h \in P^{\vee}.
 \end{align*}

Multiplying $q^{-h+\sum_{i=2}^n k_i}$ and using the linear
independence of $\fz$, we obtain $C_{\zz, h, \ee_2} =0 $ for all
$\zeta \in \Omega, \ h \in P^{\vee}$.

\end{proof}

We are now ready to prove the {\it triangular decomposition} for $U_q(\gg)$.
\begin{theorem} \label{tri decom}
  There is a $\mathbb{C}(q)$-linear isomorphism
  $$U_q(\gg) \cong U_q^{-} \ot U_q^{0} \otimes U_q^{+}.$$
\end{theorem}

\begin{proof}
Let $\{\fz ~|~ \zz \in \Omega \}$ , $\{\qh k_{\ee} ~|~ h \in
P^{\vee}, ~ \ee \in \Omega^{'} \}$, and $\{e_{\tau} ~|~ \tau \in
\Omega \}$ be monomial bases of $U_q^{-}$, $U_q^{0}$ and $U_q^{+}$
respectively, where $\Omega$ and $\Omega'$ are the index sets as in the proof for Lemma \ref{half decom}. It suffices to show that the elements
$\fqke$ ($\zz, ~ \ta \in \Omega, ~ h \in P^{\vee}, \ee \in
\Omega^{'}$) are linearly independent over $\mathbb{C}(q)$.

  Suppose
  \begin{equation*}
    \begin{aligned}
      \sum_{\zz, h, \ee,  \ta} C_{\zz,  h, \ee, \ta} \fqke = 0 \ \ \ \ \text{for some } C_{\zz,  h, \ee,  \ta} \in \mathbb{C}(q).
    \end{aligned}
  \end{equation*}
The root space decomposition of $U_q(\gg)$ yields
   \begin{equation*}
    \begin{aligned}
      \sum_{ \substack { h,\  \ee, \\ \deg \fz +\deg \et =\gamma}} C_{\zz,  h, \ee,  \ta} \fqke = 0 \ \ \ \ \text{for all } \gamma \in Q .
    \end{aligned}
  \end{equation*}

  Using the partial ordering on $\mathfrak{h}_{\bar{0}}^*$, we can
choose $\alpha=\deg \fz$ and $\beta=\deg \et$, which are minimal and
maximal, respectively, among those for which $\alpha + \beta =\gamma$ and $C_{\zz,  h, \ee, \ta}$
is nonzero. If $\alpha= -\sum m_i \alpha_i$, set
$h_{\alpha}=\sum m_i k_{i+1}$, and if $\beta= \sum
n_i \alpha_i$, set $h_{\beta}=\sum n_i k_{i+1}$. The term
of degree $(0,\beta)$ in $\Delta(\et)$ is $q^{-h_{\beta}} \ot \et$
and the term of degree $(\alpha,0)$ of $\Delta(\fz)$ is $\fz \ot
q^{h_{\alpha}}$.

Since the terms of degree $(\alpha,\beta)$ of $\sum C_{\zz, h, \ee,\tau}
\Delta(f_{\zeta} q^h k_{\eta} e_{\tau})$ must sum to zero, we have
 \begin{equation*}
    \begin{aligned}
      \sum_{\substack { \deg \fz =\alpha, \\ \deg \et =\beta, \\h, \, \eta=(a_1, \cdots, a_n)}}
      \sum_{(j_1,\ldots,j_n) \in \Omega' \atop j_i \leq a_i, \ i \in J}  C_{\zz,  h, \ee,  \ta} \fz q^{h} \Big(\prod_{i=1}^{n}
      q^{-(a_i -j_i)k_i-h_{\beta}}k_{\bar{i}}^{j_i} \Big) \ot q^{h_{\alpha}+h} \Big( \prod_{i=1}^n q^{j_i k_i}k_{\bar{i}}^{a_i -j_i} \Big) \et
      =0.
     \end{aligned}
  \end{equation*}

The elements $\fz q^{h} \big(\prod_{i=1}^{n} k_{\bar{i}}^{j_i}\big)$
are linearly independent for $\zeta \in \Omega, \ h \in P^{\vee}, \ (j_1, \cdots, j_n)
\in  \Omega'$ by Lemma \ref{half decom}. By the similar argument in the proof for Lemma \ref{half decom},  we obtain
  \begin{align*}
      \sum_{\deg \et =\beta}  C_{\zz,  h, \ee,  \ta}
       \et =0 \ \ \text{for all} \ h \in P^{\vee}, \zeta \in \Omega, \ \text{and} \ \ee \in \Omega'.
  \end{align*}

Using the linear independence of $e_{\tau}$, we conclude that
$C_{\zz, h, \ee, \ta}=0$ for all $\zeta \in \Omega, h \in P^{\vee},
\eta \in \Omega', \ \text{and} \ \tau \in \Omega$, as desired.
\end{proof}

\vskip 5mm
%%%%%%%%%%%%%%%%%%%%%%%%%%%%%%%%%%%%%%%%%%%%%%%%%%%%%%%%%%%%%%%%%%%%%%%%%%%%%%%%%%%%%%%%%%%%%%%%%%%%%%%%%%%%%%%%%%

\section{The quantum Clifford Superalgebra $\Cliff_q(\lambda)$} \label{section_clifford}

We first introduce some notation that will be used in this section
only. Let $\K$ be a field of zero characteristic and $A$ be an
associative $\K$-algebra. Denote by ${\rm Mat}_ {n}(A)$ the
associative $\K$-algebra of $n \times n$ matrices with entries in
$A$. If $A$ is a superalgebra, then ${\rm Mat}_ {n}(A)$ is a
superalgebra as well by setting ${\rm Mat}_ {n}(A)_{\bar{i}} = {\rm
Mat}_ {n}(A_{\bar{i}})$. By ${\rm sMat}_{n | n}(\K)$ we denote the
associative superalgebra of $2n \times 2n$ matrices $\left(
\begin{array}{cc} A& B \\ C & D
\end{array} \right)$ , where $A$, $B$, $C$, and $D$ are in ${\rm Mat}_ {n}(\K)$ and
$${\rm sMat}_{n,n}(\K)_{\bar{0}} =  \left\{ \left( \begin{array}{cc} A& 0 \\ 0 & D
\end{array} \right) \right\}, \quad {\rm sMat}_{n,n}(\K)_{\bar{1}} =  \left\{ \left( \begin{array}{cc} 0& B \\ C & 0
\end{array} \right) \right\} .$$
Let $Q_n(\K)$ be the subsuperalgebra of ${\rm sMat}_{n | n}(\K)$ with elements $ \left( \begin{array}{cc} A& B \\ B & A
\end{array} \right)$. In particular, $Q_n(\K)_{\bar{0}} = Q_n(\K)_{\bar{1}} = {\rm Mat}_ {n}(\K)$. There are
$\K$-superalgebra isomorphisms
$$
{\rm Mat}_r \left( {\rm sMat}_{1 | 1} (\K )\right) \cong {\rm sMat}_{r | r} (\K ), \;
{\rm Mat}_r (Q_1 (\K)) \cong Q_r(\K).
$$
Note that if $\K = \C$, then the superalgebra $Q_n(\C)$ coincides
with $\gg$ as a complex vector space. Another example of a
$\K$-superalgebra is any extension $\K (\alpha)$ of $\K$ of degree
$2$ considering $\alpha$ as an odd element. If $\alpha^2 = \beta \in \K$ we will
denote $\K (\alpha)$ by $\K (\sqrt{\beta})$.

In this section, we set $\F = \C(q)$. For every $\lambda  \in P$ we
define $I^q(\lambda)$ to be the left ideal of $U_q^0$ generated by
$q^h - q^{\lambda(h)}1$, $h \in P^{\vee}$. Set
$\mbox{Cliff}_q(\lambda):= U_q^0/I^q(\lambda)$. We may consider
$\mbox{Cliff}_q(\lambda)$ as the associative $\F$-algebra generated
by the identity ${\bf 1} = 1 + I^q(\lambda)$ and
$t_{\bar{i}}:=k_{\bar{i}} + I^q(\lambda)$ satisfying the relations
$$t_{\bar{i}}t_{\bar{j}} + t_{\bar{j}}t_{\bar{i}} =
\delta_{ij}\dfrac{2(q^{2\lambda_i} - q^{-2
\lambda_i})}{q^2-q^{-2}}{\bf 1}, \; i,j = 1,...,n.$$ Furthermore,
$\mbox{Cliff}_q(\lambda)$ has an obvious $\Z_2$-grading (and thus a
superalgebra structure) by assuming that $t_{\bar{i}}$ are odd. More
precisely, ${\rm Cliff}_q (\lambda)_{\bar{0}}$ is spanned by ${\bf
1}$ and the monomials $t_{\bar{i}_1}...t_{\bar{i}_{2k}}$ of even
degree, while ${\rm Cliff}_q (\lambda)_{\bar{1}}$ is spanned by
those of odd degree. In this section we will describe the structure
of $\mbox{Cliff}_q(\lambda)$ and will classify its irreducible
modules. Because of its superalgebra structure,
$\mbox{Cliff}_q(\lambda)$ has both $\Z_2$-graded and nongraded
modules and both cases will be addressed.

The results in this section may be derived from more general
statements about quadratic forms and Clifford superalgebras over
arbitrary fields (see, for example, \cite{Lam} and \cite{Sh}). For the
sake of completeness we will give an outline of the proofs. The
results and the proofs in this section will also help us to describe
explicitly the action of  $U_q^0$ on the highest weight vectors of
an irreducible highest weight module over $U_q(\gg)$. This is
demonstrated in Example \ref{cliff_ex} for the case $n=3$ and
$\lambda = (4,2,1)$ .

In this section, we fix $V := \bigoplus_{i=1}^n \F t_{\bar{i}}$ and
$\Lambda := (\Lambda_1,..., \Lambda_n) \in \F^n$ and denote by
$B_{\Lambda} : V \times V \to \F$  the symmetric bilinear form
defined by $B_{\Lambda}  (t_{\bar{i}}, t_{\bar{j}}) = \delta_{ij}
\Lambda_i$. Let $\mbox{Cliff}_q(\Lambda)$ be the unique up to
isomorphism Clifford algebra associated to $V$ and $B_{\Lambda}$. If
$\Lambda_i = \dfrac{q^{2\lambda_i} - q^{-2 \lambda_i}}{q^2-q^{-2}}$,
then  we have $\mbox{Cliff}_q(\Lambda) \simeq \mbox{Cliff}_q
(\lambda)$.

Define $V(\Lambda):=V/\ker B_{\Lambda}$, where $\ker B_{\Lambda}:=
\{v \in V \; | \; B_{\Lambda}(v, u)= 0, \mbox{ for every }u \in V\}$
and denote by $\beta_{\Lambda}$ the restriction of $B_{\Lambda}$ on
$V(\Lambda)$.  Let  $N_{\Lambda}= \{ i \; | \; \Lambda_i \neq 0\}$,
$Z_{\Lambda}= \{ j \; | \; \Lambda_j = 0\}$,  and $ | \Lambda | = \#
N_{\Lambda}$.  Set $\Lambda_N := (\Lambda_{i_1},...,\Lambda_{i_{|
\Lambda |}}) $, $0_{Z} := (\Lambda_{j_1},...,\Lambda_{j_{n-|
\Lambda|}}) = (0,...,0) $,  where $N_{\Lambda} = \{i_1,...,i_{|
\Lambda |} \}$,   $Z_{\Lambda} = \{j_1,...,j_{n- | \Lambda |} \}$,
and $i_1<...<i_{| \Lambda |}$.  It is clear that $\ker B_{\Lambda} =
\bigoplus_{j \in Z_{\Lambda}} \F t_{\bar{j}}$ and that
$\mbox{Cliff}_q(\Lambda_N)  = \bigoplus_{i \in N_{\Lambda}} \F
t_{\bar{i}}$ is the Clifford algebra corresponding to $(V(\Lambda),
\beta_{\Lambda})$. Furthermore,
$$\mbox{Cliff}_q(\Lambda) \simeq \mbox{Cliff}_q(\Lambda_N)
\otimes_{\F} \mbox{Cliff}_q(0_{Z}) \simeq \mbox{Cliff}_q(\Lambda_N)
\otimes_{\F} \bigwedge \ker B_{\Lambda}. $$ Here $\bigwedge W$
denotes the exterior algebra of the vector space $W$. Thanks to the
above isomorphisms every $\mbox{Cliff}_q(\Lambda) $-module can be
considered as a $\mbox{Cliff}_q(\Lambda_N) $-module under the
embedding $\mbox{Cliff}_q(\Lambda_N) = \mbox{Cliff}_q(\Lambda_N)
\otimes_{\F} 1 \to \mbox{Cliff}_q(\Lambda_N) \otimes_{\F}
\mbox{Cliff}_q(0_{Z})$. The class  $\overline{\Delta(\Lambda)}$ of
$\Delta(\Lambda) = \Pi_{i \in N_{\Lambda}} \Lambda_i$ in $\dot{\F} /
\dot{ \F}^{2}$ is called the {\it discriminant} of $(V,
B_{\Lambda})$.

The following lemma is standard and the proof is left to the reader.

\begin{lemma} \label{lambda_zero}
Let $M$ be an irreducible ${\rm Cliff}_q(\Lambda)$-module. Then $M$
is an irreducible $\mbox{\rm Cliff}_q(\Lambda_N)$-module and
$t_{\bar{i}} v = 0$ for every $i \in Z_{\Lambda}$. Conversely, if
$M_0$ is an irreducible $\mbox{\rm Cliff}_q(\Lambda_N)$-module then
$M_0$ considered as a ${\rm Cliff}_q(\Lambda)$-module with trivial
action of ${\rm Cliff}_q(0_Z)$ is irreducible as well.

\end{lemma}

Since our goal in this section is to classify the irreducible representations of
${\rm Cliff}_q(\Lambda)$, thanks to the above lemma, we may assume
that $\Lambda_i$ are nonzero. So, for simplicity {\it we fix
$Z_{\Lambda}= \emptyset$, and thus $B_{\Lambda} = \beta_{\Lambda}$ and
$V(\Lambda) = V$, in all statements preceding Corollary
\ref{cliff_mod}}.

Recall that a vector $v$ in $V$ is called $\beta_{\Lambda}-${\it
isotropic} (or simply {\it isotropic}) if $\beta_{\Lambda} (v,v) =
0$. A subspace $W$ of $V$ is $\beta_{\Lambda}-${\it isotropic
subspace} if $\beta_{\Lambda} (u,w) = 0$ for every $u$ and $w$ in $W$. A
subspace $W$ of $V$ is {\it anisotropic} if it contains no nonzero
$\beta_{\Lambda}-$isotropic vector. An isotropic subspace $W$ of $V$
is {\it maximal isotropic} if there is no larger
$\beta_{\Lambda}$-isotropic subspace containing $W$.

\begin{lemma} \label{Wittdec}
Let $W$ be an isotropic subspace of $V$. Then there exists an
isotropic subspace $W^*$ and a subspace $Z$ of $V$ such that
\begin{equation*}
\begin{aligned}
& V = Z \oplus W \oplus W^*, \quad \dim W = \dim W^*, \\
& \beta_{\Lambda}(z,w) = \beta_{\Lambda}(z,w^*) = 0 \ \ \text {for
every} \  z \in Z, w \in W, w^* \in W^*.
\end{aligned}
\end{equation*}
Moreover, there exist bases $\{ w_1,...,w_m \}$ and $\{
w_1^*,...,w_m^* \}$ of $W$ and $W^*$, respectively, such that
$\beta_{\Lambda}(w_i, w_j^*) = \delta_{ij}$.
\end{lemma}
\begin{proof} The lemma follows by induction on $\dim W$. If $\dim W = 1$,
then $W^*$ is spanned by $w_1^* = x - \frac{1}{2} \beta_{\Lambda}
(x,x) w_1$, where $x \in V$ is arbitrarily chosen so that
$\beta_{\Lambda} (w_1, x) = 1$. Then we define $Z$ to be
$$
Z = \{ z \in V \; | \;  \beta_{\Lambda} (z,w_1) =  \beta_{\Lambda} (z,w_1^*) = 0 \}.
$$
For the complete proof, see \cite[Lemma 1.3]{Sh}.
\end{proof}

The decomposition  $V =  Z \oplus W \oplus W^*$ in Lemma
\ref{Wittdec} is called a {\it weak Witt decomposition of $V$}. For
any weak Witt decomposition $V =  Z \oplus W \oplus W^*$, we denote by $\Cliff (\Lambda_Z)$ the Clifford algebra corresponding to $(Z, \beta_{\Lambda | Z})$. If $V = Z \oplus
W \oplus W^*$ is a weak Witt decomposition for which $Z$ is
anisotropic (or, equivalently, $W$ is maximal isotropic) we call it
a {\it Witt decomposition}.  We may  identify $W^*$ with the dual
space of $W$ via the nondegenerate form $\beta_{\Lambda}$. If $V = Z \oplus
W \oplus W^*$ is a  Witt decomposition, the
dimension of $W$ is an invariant of $(V, \beta_{\Lambda})$ (see
\cite[Lemma 1.4]{Sh}) and is known as the {\it Witt index} of the
form $\beta_{\Lambda}$.  We say that the Witt index is {\it maximal}
if $\dim Z \leq 1$. Recall that if the ground field is $\C$, the
Witt index is always maximal. In the case of arbitrary $\F$ though,
the Witt index is generally not maximal as we verify in Lemma
\ref{Wittindex}. In order to find a Witt decomposition and the Witt
index  of $(V, \beta_{\Lambda})$ we need some preparatory
statements.

\begin{lemma} \label{cliff_z}
Let $V =  Z \oplus W \oplus W^*$  be a weak Witt decomposition and
let $m = 2^{\dim W}$. Then ${\rm Cliff}_q(\Lambda) \cong {\rm Mat}_m
({\rm Cliff}_q(\Lambda_Z))$. Moreover, we have
\begin{equation*}
{\rm Cliff}_q(\Lambda)_{\bar{0}} \cong
\begin{cases}
{\rm Mat}_m ({\rm Cliff}_q(\Lambda_Z)_{\bar{0}}) \ \ & \text{if} \ Z
\neq 0, \\
{\rm Mat}_{m/2} ( \F ) \oplus {\rm Mat}_{m/2} ( \F ) \ \ & \text{if}
\ Z = 0.
\end{cases}
\end{equation*}
\end{lemma}

\begin{proof}
For the complete proof, see \cite[Theorem 2.6]{Sh}. The proof
follows by induction on $\dim W$. We sketch the proof for $\dim W =
1$. In this case there is an isomorphism $\Psi : {\rm
Cliff}_q(\Lambda) \to {\rm Mat}_m ({\rm Cliff}_q(\Lambda_Z))$
defined by its restriction $\Psi_{| V}$ on $V$:
$$
z + rw_1 + sw_1^* \mapsto \left( \begin{array}{cc} z& r \\ s & -z
\end{array} \right).
$$
Notice that if $Z \neq 0$, $\Psi$ is not necessarily parity preserving. In such a case we choose the
isomorphism $\Theta :  {\rm Cliff}_q(\Lambda) \to {\rm Mat}_m ({\rm
Cliff}_q(\Lambda_Z))$ defined by $\Theta (\alpha) = D^{-1} \Psi
(\alpha) D$, where $D =  \left( \begin{array}{cc}  g&  0
\\ 0 & 1
\end{array} \right)$ and any $g \in Z$ with $\beta_{\Lambda}( g, g ) \neq 0$. \end{proof}

\begin{lemma} \label{q-legandre}
The nondegenerate Legendre's equation always has a nontrivial
solution in $\F${\rm :} for every nonzero $A,B$, $C$ in $\F$, there
exist $X,Y,Z \in \F$ with $(X,Y,Z) \neq (0,0,0)$ such that $AX^2 +
BY^2 + CZ^2=0$.
\end{lemma}
\begin{proof}
We modify the proof of the classical Legendre's Theorem (see, for
example, \cite[\S 17.3]{IR}). We first assume that $A,B,C,X,Y,Z$ are
polynomials in $\C[q]$, where $A,B$, $C$ are square free. We
may fix $C=-1$, since if $(X,Y,Z)$ is a solution of $ACX^2 + BCY^2 =
Z^2$ then $(X,Y,\sqrt{-1}\frac{Z}{C})$ is a solution of $AX^2 + BY^2
+ CZ^2=0$. We prove that  $AX^2 + BY^2 = Z^2$ has a nontrivial solution by
induction on $N:=\max \{ \deg A, \deg B \}$.

If $N=0$; i.e., $A$ and $B$ are constant polynomials, then $AX^2 +
BY^2 = Z^2$ has a solution (constant polynomials). Assume that $\deg
B \leq \deg A$ and $\deg A \geq 1$. Recall that every polynomial $R
\in \C[q]$ is a quadratic residue modulo any square free
polynomial $S$. Indeed, if $S$ is constant, our assertion is obvious.
Otherwise, let $S (q) = \Pi_{i=1}^r (q-z_i)$ with $z_i\neq z_j$, and
let $y_i \in \C$ be such that $y_i^2 = R(z_i)$. Then $y_i^2 \equiv R
\, ( \mbox{mod}\, (q-z_i))$. Using the Chinese Remainder Theorem, we
find $y \in \C[q]$ for which $y \equiv y_i  \,( \mbox{mod}\,
(q-z_i))$. But then $y^2 \equiv R  \,( \mbox{mod}\, (q-z_i))$ and
thus $y^2 \equiv R \, ( \mbox{mod}\, S)$.

We fix $C_1$ with $\deg C_1 < \deg A$ such that $C_1^2 \equiv B \,
(\, \mbox{mod} A)$. Then $C_1^2 - B = AT = AA_1M^2$ for some square
free polynomial $A_1$. Since $\deg A + \deg A_1 \leq \deg(A A_1 M^2)
= \deg (C_1^2 - B) < 2 \deg A$, we have $0 \leq \deg A_1 < \deg A$.
Now we observe that if $(X_1, Y_1, Z_1)$ is a solution of $A_1X^2 +
BY^2 = Z^2$, then $(A_1X_1M, C_1Y_1+Z_1, Z_1C_1 + BY_1)$ is a
solution of $A X^2 + BY^2 = Z^2$. Using the induction hypothesis, we
complete the proof.
\end{proof}

\begin{remark} Lemma \ref{q-legandre} may be proved with a standard
algebro-geometric argument using dimensions, see, for example,
\cite[Exercise 11.6]{Har}. The lemma is also a particular case of
the following Theorem of Tsen-Lang: if $K$ is a field of
transcendence degree $n$ over an algebraically closed field $k$,
then any quadratic form over $K$ of dimension bigger than $2^n$ is
isotropic. For details, see \cite[Chapter XI]{Lam}.
\end{remark}

{\it In what follows, we assume  $\Lambda_i = \dfrac{q^{2\lambda_i} -
q^{-2 \lambda_i}}{q^2-q^{-2}}$}. For simplicity, we will write $
\beta_{\lambda}$, $|\lambda|$, and $\Delta(\lambda)$ for
$\beta_{\Lambda}$, $|\Lambda|$,  and $\Delta (\Lambda)$,
respectively.
The following technical lemma can be easily verified.

\begin{lemma}
Define an equivalence relation $\sim$ in $\{ \lambda_i \; | \; i
=1,...,n\}$ by $\lambda_i \sim \lambda_j$ if $\lambda_i^2 =
\lambda_j^2$ and denote by $o(\lambda_i)$ the orbit of $\lambda_i$
relative to $\sim$. Then $\overline{\Delta(\lambda)} =\bar{1}$ {\rm
(}or, equivalently, $\Delta (\lambda)$ is a square in $\F${\rm )} if
and only if the orbit $o(\lambda_i)$ of every $\lambda_i \neq \pm 1$
contains even number of elements.

\end{lemma}

\begin{lemma} \label{Wittindex}
The space $V$ is anisotropic if and only if $\dim V = 1$ or   $\dim
V = 2$ and $\overline{\Delta (\lambda)} \neq \overline{1}$.  If  $V$
is isotropic, there is a Witt decomposition $V = Z \oplus W \oplus
W^*$ of $V$ such that
\begin{itemize}
\item[{\rm (1)}] $\dim W = k$ if $\dim V = 2k+1$, $k\geq 1$ {\rm (}maximal Witt index{\rm )};

\item[{\rm (2)}]  $\dim W = k-1$ if $\dim V= 2k$ and  $\overline{\Delta (\lambda)}  \neq \overline{1}$;

\item[{\rm (3)}] $\dim W = k$ if $\dim V = 2k$ and $\overline{\Delta (\lambda)} = \overline{1}$  {\rm (}maximal Witt index{\rm )}.
\end{itemize}
In particular, if $\lambda_1 > \lambda_2 >...>\lambda_n >0$, then
$\dim W = \left[ \frac{n-1}{2}\right]$.
\end{lemma}

\begin{proof}
The proof consists of several steps.

{\it Step 1:  The case $\dim V = 1$}. This case is straightforward.

{\it Step 2:  The case $\dim V = 2$}. In this case,
$v = a_1 t_{\bar{1}} + a_2
t_{\bar{2}}$  is $\beta_{\lambda}$-isotropic if and only if $a_1^2 \Lambda_1 + a_2^2 \Lambda_2 =
0.$ The latter equation has a solution for $a_1$ and $a_2$ if and
only if $\frac{\Lambda_1}{\Lambda_2}$ is a square (or equivalently,
$\Lambda_1 \Lambda_2$ is a square).

{\it Step 3: If $\dim V  \geq 3$, then
$$
V \cong \F w \oplus \F w^* \oplus \F v_3 \oplus ... \oplus \F v_n,
$$
where
\begin{equation*}
\begin{aligned}
& \beta_{\lambda} (w, w)  = \beta_{\lambda} (w^*, w^*)  =
\beta_{\lambda} (w,v_i)  = \beta_{\lambda} (w^*,v_i)=  0 \ \
\text{for} \ i \geq 3, \\
& \beta_{\lambda} (w, w^*)  = 1, \  \ \beta_{\lambda} (v_3, v_3) =
\Lambda_1 \Lambda_2 \Lambda_3, \ \ \beta_{\lambda} (v_i, v_i) =
\Lambda_i \ \ \text {if} \ i \geq 4.
\end{aligned}
\end{equation*}}

Let us first consider the case $\dim V = 3$. We use Lemma
\ref{q-legandre} to find $w = x_1 t_{\bar{1}} + x_2 t_{\bar{2}} +
x_3 t_{\bar{3}}$ such that $\beta_{\lambda} (w,w)  = 0$. Applying
Lemma \ref{Wittdec} to $W = \F w$, we find $w^* = y_1 t_{\bar{1}} +
y_2 t_{\bar{2}} +  y_3 t_{\bar{3}}$ and $z = z_1 t_{\bar{1}} + z_2
t_{\bar{2}} +  z_3 t_{\bar{3}}$ such that
$$\beta_{\lambda}
(w^*,w^*)  =  \beta_{\lambda} (w^*,z)  = \beta_{\lambda} (w,z)  = 0,
\ \ \beta_{\lambda} (w,w^*)  = 1.$$
The choice of $z$ is unique up
to a multiplication by a nonzero constant in $\F$. A simple
calculation shows that $z_i$ may be chosen as follows
\begin{eqnarray*}
z_1 & = &  \sqrt{-1}\Lambda_2 \Lambda_3 (x_2y_3 - x_3 y_2), \\
z_2 & = & \sqrt{-1}\Lambda_1 \Lambda_3 (x_3y_1 - x_1 y_3), \\
z_3 & = & \sqrt{-1}\Lambda_1 \Lambda_2 (x_1y_2 - x_2 y_1).
\end{eqnarray*}
Then one can easily verify that $\beta_{\lambda} (z,z) = \Lambda_1
\Lambda_2 \Lambda_3$.

In the case $\dim V  > 3$, write $V = \F t_{\bar{1}} \oplus \F
t_{\bar{2}} \oplus \F t_{\bar{3}} \oplus \left( \bigoplus_{i \geq 4}
\F t_{\bar{i}}\right)$. Fix $w,w^*,z \in \F t_{\bar{1}} \oplus \F
t_{\bar{2}} \oplus \F t_{\bar{3}}$ as above, and set $v_3 = z$ and
$v_i = t_{\bar{i}}$ for $i\geq 4$.

{\it Step 4: If $\dim V \geq 3$, then $V$ has a Witt decomposition
$$
V \cong Z \oplus W \oplus W^*,
$$
where
\begin{equation*}
\dim Z = \begin{cases} 0 \ \ & \text {if $\dim V $ is even and
$\Lambda_1\Lambda_2...\Lambda_n$ is a square}, \\
1 \ \ & \text {if $\dim V $ is odd}, \\
2 \ \ & \text {if $\dim V $ is even and
$\Lambda_1\Lambda_2...\Lambda_n$ is not a square.}
\end{cases}
\end{equation*} }

This follows from an inductive argument using Step 1, Step 2, and
Step 3.
\end{proof}

\begin{lemma} \hfill \label{low_cliff}
\begin{itemize}
\item[\rm{(1)}] Assume that $\dim V = 1$. Then
\begin{equation*}
{\rm Cliff}_q(\lambda) \cong
\begin{cases}
Q_{1}(\F) \ \ & \text {if $\overline{\Delta(\lambda)} = \bar{1}$
{\rm (}equivalently, $\Lambda_1$ is a square in $\F${\rm )}}, \\
\F (\sqrt{\Lambda_1}) \ \ & \text {if $\overline{\Delta(\lambda)}
\neq \bar{1}$ {\rm (}equivalently, $\Lambda_1$ is not a square in
$\F${\rm )}}.
\end{cases}
\end{equation*}

%$\bullet$ if $\overline{\Delta(\lambda)} = \bar{1}$ {\rm (}or,
%equivalently, $\Lambda_1$ is a square in $\F${\rm )}, ${\rm
%Cliff}_q(\lambda) \cong Q_1 (\F)$;
%$\bullet$ if $\overline{\Delta(\lambda)} \neq \bar{1}$ {\rm (}or, equivalently, $\Lambda_1$ is not a square in $\F${\rm )}, ${\rm Cliff}_q(\lambda) \cong \F (\sqrt{\Lambda_1})$.

\item[\rm{(2)}] Assume that $\dim V  = 2$. Then
${\rm Cliff}_q(\lambda) \cong {\rm Mat}_2 (\F)$
as {\rm (}nongraded{\rm )} algebras and
${\rm Cliff}_q(\lambda)_{\bar{0}} \cong {\rm Cliff}_q(\Lambda_1 \Lambda_2)$.
\end{itemize}
\end{lemma}
\begin{proof}
The case (1) corresponds to the ``classical case'' (Clifford
superalgebra over $\C$) and can be easily verified.

(2) Let $A = {\rm Cliff}_q(\lambda)$. Then $A$ is a quaternion
algebra over $\F$. Since it is not a division algebra, by
Wedderburn's Theorem, we have $A \cong {\rm Mat}_2 (\F)$ (see
\cite[Theorem 2.7]{Lam} for details). The isomorphism $A_{\bar{0}}
\cong {\rm Cliff}_q(\Lambda_1 \Lambda_2)$ is straightforward.
\end{proof}

\begin{remark}
The superalgebraic  structure of ${\rm Cliff}_q(\lambda)$ for $\dim
V = 2$ is ``explicit'' only when
$\overline{\Delta(\lambda)} = \bar{1}$. In this case, one can show
that  ${\rm Cliff}_q(\lambda) \cong {\rm sMat}_{1|1} (\F)$.
\end{remark}

We are now ready to describe the superalgebra structure of ${\rm Cliff}_q (\lambda)$.

\begin{proposition} \hfill \label{cliff_structure}
\begin{itemize}
\item[\rm{(1)}] If $n$ is even,
then ${\rm Cliff}_q(\lambda) \cong {\rm Mat}_{r } (A)$, where $A =
\Cliff_q ((\Delta(\lambda), 1))$ and  $r = 2^{\frac{n}{2} -1}$.
Furthermore,  ${\rm Cliff}_q(\lambda) \cong {\rm Mat}_{2r } (\F)$ as (nongraded)
algebras and
\begin{equation*}
{\rm Cliff}_q(\lambda)_{\bar{0}} \cong
\begin{cases}{\rm Mat}_{r} (\F) \oplus {\rm Mat}_{r} (\F) \ \ &
\text{if} \ \overline{\Delta(\lambda)} = \bar{1}, \\
{\rm Mat}_{r} (\F (\sqrt{\Delta(\lambda)})) \ \ & \text{if} \
\overline{\Delta(\lambda)} \neq \bar{1}.
\end{cases}
\end{equation*}

\item[\rm{(2)}] If $n$ is odd, then ${\rm Cliff}_q(\lambda) \cong {\rm Mat}_{r } (B)$,
where $B = {\rm Cliff}_q (\Delta(\lambda)) $ and  $r = 2^{\frac{n-1}{2}}$.
Furthermore,
\begin{equation*}
\begin{cases}
{\rm Cliff}_q(\lambda) \cong Q_r (\F), \quad {\rm
Cliff}_q(\lambda)_{\bar{0}} \cong {\rm Mat}_{r} (\F) \ \ & \text{if}
\ \overline{\Delta(\lambda)} = \bar{1}, \\
{\rm Cliff}_q(\lambda) \cong {\rm Mat}_r (\F
(\sqrt{\Delta(\lambda)})), \quad {\rm Cliff}_q(\lambda)_{\bar{0}}
\cong {\rm Mat}_{r} (\F) \ \ & \text{if} \
\overline{\Delta(\lambda)} \neq \bar{1}.
\end{cases}
\end{equation*}
In particular, ${\rm Cliff}_q(\lambda) $ is a simple superalgebra
which is isomorphic to

$\bullet$  a direct sum of two isomorphic simple algebras if $n$ is odd and $\overline{\Delta(\lambda)} = \bar{1}$;

$\bullet$ a simple algebra otherwise.

\end{itemize}
\end{proposition}
\begin{proof}

%As mentioned earlier, we assume that $| \lambda | = n$.

We first consider the case when $n$ is even and let  $r =
2^{\frac{n}{2} -1}$. If $\Lambda_1...\Lambda_n$ is a square in $\F$,
then (1) is proved by Lemma \ref{Wittindex} (3) and Lemma
\ref{cliff_z}. Now if $\overline{\Delta(\lambda)} \neq \bar{1}$, by
Lemma  \ref{cliff_z} and Step 3 in the proof of Lemma
\ref{Wittindex}, we have ${\rm Cliff}_q(\Lambda) \cong {\rm
Mat}_r(A)$, where $A = {\rm Cliff}_q(\Lambda_1...\Lambda_{n-1},
\Lambda_n)$.  We now  apply Lemma \ref{low_cliff} (1),(2) and prove
(1).

Next, assume that $n$ is odd and let  $r = 2^{\frac{n - 1}{2} }$. By
Lemma \ref{cliff_z} and Step 3 in the proof of Lemma
\ref{Wittindex}, we have ${\rm Cliff}_q(\lambda) \cong {\rm
Mat}_r(B)$, where $B$  is the $2$-dimensional Clifford superalgebra
${\rm Cliff}_q(\Lambda_1...\Lambda_n)$. We use Lemma \ref{low_cliff}
(1) to complete the proof. \end{proof}

In the statement of the following corollary we allow $\lambda_i$ to be zero for
some $i$. Recall that $| \lambda |$ is the number of nonzero
$\lambda_i$.  We also set $\lambda_N :=
(\lambda_{i_1},...,\lambda_{i_{| \lambda |}})$ where $N_{\lambda} =
\{i_1,...,i_{| \lambda |} \}$ and $i_1 <...< i_{| \lambda |}$.

\begin{corollary} \label{cliff_mod}
Every $\Z_2$-graded ${\rm Cliff}_q(\lambda_N) $-module  is
completely reducible. Furthermore, the superalgebra ${\rm
Cliff}_q(\lambda) $ has up to isomorphism
\begin{itemize}
\item[\rm{(1)}] two simple modules $E^q(\lambda)$ and $\Pi (E^q(\lambda))$
of dimension $2^{k-1} | 2^{k-1}$ if $| \lambda| = 2k$  and
$\overline{\Delta(\lambda)} = \bar{1}${\rm ;}

\item[\rm{(2)}] one simple module $E^q (\lambda) \cong \Pi (E^q (\lambda))$
of dimension $2^{k} | 2^{k}$ if $| \lambda| = 2k$ and
$\overline{\Delta(\lambda)} \neq \bar{1}$ {\rm (}in particular, if
$\lambda_1 >....> \lambda_{2k} >0${\rm)}{\rm ;}

\item[\rm{(3)}] one simple module $E^q (\lambda) \cong \Pi (E^q (\lambda))$
of dimension $2^{k} | 2^{k}$ if $|\lambda| = 2k+1$.
\end{itemize}
\end{corollary}
\begin{proof}  Thanks to Lemma \ref{lambda_zero}, we may assume that
$\lambda_i \neq 0$; i.e., $|\lambda|= n$. The category of all
$\Z_2$-graded ${\rm Cliff}_q(\lambda) $-modules is
equivalent to the category of all nongraded  ${\rm
Cliff}_q(\lambda)_{\bar{0}}$-modules. Indeed, the reverse
correspondence is obtained by $$V_0 \mapsto {\rm Cliff}_q(\lambda)
\otimes_{{\rm Cliff}_q(\lambda)_{\bar{0}}} V_0.$$ The corollary
follows from Proposition \ref{cliff_structure} and the
characterization of the simple and indecomposable (nongraded)
modules of  ${\rm Mat}_{r } (\F) \oplus {\rm Mat}_{r } (\F) $, ${\rm
Mat}_{r } (\F) $, and  ${\rm Mat}_{r }(\F
(\sqrt{\Delta(\lambda)}))$. (This characterization may be found, for
example, in \cite[Chapter XVII]{Lang}.)  \end{proof}

\begin{example} \label{cliff_ex}
Let $n=3$ and $\lambda = (4,2,1)$. We describe the action of
$t_{\bar{i}}$ ($i=1,2,3$) on $E^q (\lambda)$.  We have
$$\Lambda_1 =(q^2 + q^{-2})(q^4 + q^{-4}), \ \ \Lambda_2 = q^2 + q^{-2}, \ \
\Lambda_3 = 1.$$ For simplicity, let $t = q^2 + q^{-2}$. We first
find a solution of Legendre's equation
\begin{equation} \label{ex_eq}
\Lambda_1 X^2 + \Lambda_2 Y^2 + \Lambda_3 Z^2 = 0
\end{equation}
We follow the proof of  Lemma \ref{q-legandre}. Let $Z = t Z'$ and
$Y = \sqrt{-1} Y'$. In order to solve the equation $(t^2 -2)X^2 + t
Z'^2 = Y'^2$ we find $C_1 \in \C[t]$ for which $C_1^2 - t$ is a
multiple of $t^2 -2$. Using the Chinese Remainder Theorem, we choose
$$C_1 = \frac{\sqrt[4]{8}}{4}(1-\sqrt{-1})t +
\frac{\sqrt[4]{2}}{2}(1+\sqrt{-1}).$$ Then we solve the equation
$A_1 X_1^2 + B Z_1^2 = Y_1^2$ for $A_1 = -
\frac{\sqrt{2}}{4}\sqrt{-1}$ and $B = t$. A solution for this is
$$(X_1, Y_1 , Z_1) = (1, \frac{\sqrt[4]{8}}{4}(1-\sqrt{-1}), 0 ).$$
Then (\ref{ex_eq}) has a solution
$$
(A_1X_1, \sqrt{-1}( Y_1 C_1 +  B Z_1), t (C_1 Z_1 + Y_1) ) = \left(
-\frac{\sqrt{2}}{4} \sqrt{-1}, \frac{\sqrt{2}}{4} t + \frac{1}{2}
\sqrt{-1}, \frac{\sqrt[4]{8}}{4}(1- \sqrt{-1})t \right).
$$
Multiplying by an appropriate constant and changing signs, we fix
the following solution of (\ref{ex_eq})
$$
w = (X, Y, Z) = (1, \sqrt{-1}t - \sqrt{2}, \sqrt[4]{2}(1 + \sqrt{-1})t).
$$
We consider $w$ as an element in $V$ relative to the basis
$\{t_{\bar{1}}, t_{\bar{2}}, t_{\bar{3}}\}$. We use Lemma
\ref{Wittdec} to find a Witt decomposition $V = \F w \oplus \F w^*
\oplus \F z$. As mentioned in the proof of Lemma \ref{Wittdec}, we
find
$$
w^* = c(X, Y, -Z),
$$
where $c = \frac{\sqrt{-1}}{4 \sqrt{2}} t^{-2}$ such that
$\beta_{\lambda} (w, w^*)=1$ and $\beta_{\lambda} (w^*, w^*)=0$.
Then, as pointed out in Step 3 of the proof of Lemma
\ref{Wittindex}, we can find
$$
z = c(t YZ, -t(t^2-2)XZ, 0)
$$
such that $\beta_{\lambda} (z, w)= \beta_{\lambda} (z, w^*)= 0$ and
$ \beta_{\lambda} (z, z)=  -\frac{1}{4} \Lambda_1 \Lambda_2 \Lambda_3 = -\frac{1}{4} t^2
(t^2-2)$. Set $\alpha=\sqrt{\Delta (\lambda)} = \sqrt{t^2 (t^2-2)}$. Using Lemma
\ref{cliff_z}, we define an isomorphism $\Theta: {\rm Cliff}_q
(\lambda) \to {\rm Mat}_2 (\F (\alpha))$ by
$$
w \mapsto \left( \begin{array}{cc} 0& \alpha^{-1} \\ 0 & 0
\end{array} \right), \ \  w^* \mapsto \left( \begin{array}{cc} 0& 0 \\
\alpha & 0 \end{array} \right),  \ \ z \mapsto \left(
\begin{array}{cc} \alpha& 0 \\ 0 & -\alpha \end{array} \right).
$$
From Proposition \ref{cliff_structure} and Corollary \ref{cliff_mod}
we find that $E^q(\la) = \F(\alpha)^{\oplus 2}$. Let $v_1$ and $v_2$ be the
standard basis vectors of the $\F(\alpha)$-vector space $E^q(\la)$,
and let $\bar{v}_i = \alpha v_i$ $(i=1,2)$. The action of
$ {\rm Cliff}_q (\lambda)$ on $E^q (\la)$ is given by
\begin{eqnarray*}
z(v_1)   =    \bar{v}_1, z(v_2)   =   - \bar{v}_2, z(\bar{v}_1) =   t^2(t^2-2) v_1, z(\bar{v}_2)  =    - t^2(t^2-2) v_2\\
w(v_1)  =   0, w(v_2)   =   (t^2(t^2-2))^{-1} \bar{v}_1, w(\bar{v}_1) =   0, w(\bar{v}_2)  =   v_1\\
w^*(v_1)  =   \bar{v}_2, w^*(v_2)  =  0, w^*(\bar{v}_1) =  t^2(t^2-2)v_2, w^*(\bar{v}_2)  =   0
\end{eqnarray*}

In order to determine the action of $t_{\bar{i}}$ ($i=1,2,3$) on
$E^q(\la)$, we need to express $t_{\bar{1}}, t_{\bar{2}},
t_{\bar{3}}$ in terms of  $z, w, w^*$. With simple computations
we find:

\begin{eqnarray*}
t_{\bar{1}} & =  &\frac{\sqrt{-1}}{4 \sqrt{2}} \frac{t^2-2}{t}w + t(t^2 - 2) w^*
+ \frac{\sqrt[4]{8} (1-\sqrt{-1})}{2} \frac{\sqrt{-1}t - \sqrt{2}}{t} z, \\
t_{\bar{2}} & = & \frac{\sqrt{-1}}{4 \sqrt{2}} \frac{\sqrt{-1}t - \sqrt{2}}{t}w
+ (\sqrt{-1}t^2 - \sqrt{2} t) w^* + \frac{\sqrt[4]{8} (\sqrt{-1} -1)}{2} \frac{1}{t} z,  \\
t_{\bar{3}} & =  & \frac{1 - \sqrt{-1}}{4 \sqrt[4]{2} t} w +
\frac{\sqrt{2}(1+\sqrt{-1})t}{\sqrt[4]{2}} w^*.
\end{eqnarray*}
\end{example}

\vskip 5mm

%%%%%%%%%%%%%%%%%%%%%%%%%%%%%%%%%%%%%%%%%%%%%%%%%%%%%%%%%%%%%%%%%%%%%%%%%%%%%%%%%%%%%%%%%%%%%%%%%%%%%%%%%%%%%%%%%

\section{Highest Weight Representation theory of $U_q(\gg)$} \label{section_highest}

   A $U_q (\gg)$-module $V^q$ is called a {\it weight module} if it admits a {\it weight space decomposition}
   \begin{equation*}
     \begin{aligned}
       V^q = \bigoplus_{\mu \in P} V_{\mu}^q, ~~\text{where}~~ V_{\mu}^q =\{ v \in V^q ~|~ q^h v = q^{\mu(h)} v ~~\text{for all}~~ h \in P^{\vee}\}.
     \end{aligned}
   \end{equation*}

For a weight $U_q(\gg)$-module $V^q$, we set $\wt V^q = \{ \lambda \in P ~|~ V^q_{\lambda} \ne 0 \}$. By the same argument as in \cite[Ch.3]{HK2002}, it can be verified
that every submodule of a weight $U_q(\gg)$-module is also a
weight module. If $\dim_{\mathbb{C}(q)} V_{\mu}^q < \infty$ for all
$\mu \in P$, then the {\it character} of $V^q$ is defined to be
    \begin{equation*}
      \begin{aligned}
        \ch V^q = \sum_{\mu \in P} (\dim_{\mathbb{C}(q)} V_{\mu}^q)~e^{\mu},
      \end{aligned}
    \end{equation*}
where $e^{\mu}$ are formal basis elements of the group algebra
$\mathbb{C}(q)[P]$ with the multiplication given by
$e^{\lambda}e^{\mu}=e^{\lambda + \mu}$ for all $\lambda, \mu \in
P$.

A weight module $V^q$ is called a {\it highest weight module} if it is generated over $U_q (\gg)$
by a finite dimensional irreducible  $U_q^{\geq 0}$-module ${\bold
v}^q$. Note that ${\bold v^q}$ also admits a weight space
decomposition.  We call a vector in ${\bold v^q}$ a {\it highest weight vector} of $V^q$. Combining Lemma
\ref{half decom} and the triangular decomposition of $U_q(\gg)$
(Theorem \ref{tri decom}), we obtain $V^q =U_q^{-} {\bold v}^q$.

%The following proposition can be verified in a similar manner in\cite[Proposition 2]{PS1}.

\begin{proposition} \label{\Uq^0-module}
If ${\bold v}^q$ is a finite dimensional irreducible $\Uq^{\geq
0}$-module with a weight space decomposition ${\bold v}^q=
  \bigoplus _{\mu \in P} {\bold v}^q_{\mu}$,
  then ${\bold v}^q$ is irreducible as a $\Uq^0$-module
  and ${\bold v}^q = {\bold v}^q_{\la}$ for some $\la \in P$. Conversely, if ${\bold v}^q$ is an irreducible $U_q ^0$-module
  on which the even part of $U_q ^0$ acts by a weight $\lambda$,
  then ${\bold v}^q$ can be endowed with the structure of an
  irreducible $U_q^{\geq 0}$-module by letting $\Uq^+$ act
  trivially on ${\bold v}^q$.
\end{proposition}
\begin{proof}
Because ${\bold v}^q$ is finite dimensional, there exists a weight
$\la \in P$ such that ${\bold v}^q_{\la} \neq 0$ and ${\bold v}^q
_{\la+\al_i} =0$ for all $i \in I$. Then we have $U_q^+ ~ {\bold
v}^q _{\la}={\bold v}^q_{\la}$ and $U_q^0 ~ {\bold v}^q_{\la}
={\bold v}^q_{\la}$. Thus ${\bold v}^q_{\la}$ is a
$U_q^{\geq 0}$-submodule of ${\bold v}^q$ and hence ${\bold
v}^q_{\la}={\bold v}^q$. The other direction is obvious from the
defining relations of $U_q(\gg)$ in Theorem \ref{defining
relations of Uqqn}.
\end{proof}

\begin{comment}
\begin{proposition} \label{action of kibar}
Let ${\bold v}^q$ be an irreducible $U^0_q$-module of weight
$\la=\la_1\ep_1+\cdots+\la_n\ep_n \in P$. If $\la_i=0$ for some $i
\in I$, then $k_{\bar i} v=0$ for all $v \in {\bold v}^q$.
\end{proposition}
\begin{proof}
Consider ${\bold w}=\{ w \in {\bold v}^q ~|~ k_{\bar i}w=0\}$. One
can easily show that ${\bold w}$ is a $U^0_q$-submodule of ${\bold
v}^q$. If there exists nonzero $v \in {\bold v}^q$ such that
$k_{\bar i}v \neq 0$; i.e., if ${\bold w} \lneq {\bold v}^q$, then
$k_{\bar i}^2 v=0$ yields $k_{\bar i} v \in {\bold w}$. Thus
${\bold w}$ is nontrivial, which is a contradiction.
\end{proof}
\end{comment}

\begin{remark} \label{Uqgeq0-module and Cliffq-module}
If ${\bold v^q}$ is a finite dimensional
irreducible $U_q^{\geq 0}$-module which generates a highest weight
module $V^q$ of highest weight $\lambda$, then, by Proposition
\ref{\Uq^0-module}, we know that ${\bold
v^q}$ is an irreducible $U_q^0$-module of weight $\lambda$.
Thus ${\bold v^q}$ is a finite dimensional irreducible module over
$\Cliff_q(\lambda)=U_q^0 /I^q(\lambda) $. Conversely, if $E^q$ is
a finite dimensional irreducible $\Cliff_q(\lambda)$-module, then
it is clear that $E^q$ is an irreducible $U^0_q$-module of weight
$\lambda$.
\end{remark}

  By Corollary \ref{cliff_mod}, we know that, up to isomorphism,
${\rm Cliff}_q(\la)$ has at most two simple modules:
$E^q(\lambda)$ and $\Pi (E^q(\la))$. The $U_q(\gg)$-module
$W^q(\la) = U_q (\gg) \otimes_{U_q^{\geq 0}} E^q(\la)$ is called
the {\it Weyl module} of $U_q(\gg)$ corresponding to $\la$
(defined up $\Pi$).

 \begin{proposition} \label{verma} \hfill
  \begin{enumerate} [{\quad \rm(1)}]

    \item $W^q(\la)$ is a free $U_q^-$-module of rank $\dim E^q(\la)$.

    \item Every highest weight $U_q(\gg)$-module with highest
    weight $\la$ is a homomorphic image of $W^q(\la)$.

    \item Every Weyl module $W^q(\la)$ has a unique maximal submodule $N^q(\la)$.
  \end{enumerate}
 \end{proposition}

 \begin{proof}
   (1) This is clear from the definition.

   (2) Let $V^{q}$ be a highest weight module with highest weight
   $\la$ generated by the irreducible $U_q^{\geq 0}$-module ${\bold
v}^q$. Because ${\bold v}^q$ is irreducible over $\Cliff_q(\la)$,
it is isomorphic to $E^q(\la)$ up to $\Pi$. Thus the map
$\phi : W^q(\la) \longrightarrow V^q$ induced by $E^q(\la) \to
{\bold v}^q$ is a surjective $U_q(\gg)$-module
homomorphism.

   (3) Since $E ^q(\la)$ is an irreducible ${\rm Cliff}_q(\la)$-module,
   any proper submodule $N^q$ of $W^q(\la)$ does not contain highest
weight vectors (the vectors in $E^q(\lambda)$). That is, $N^q$ must
lie in $\bigoplus_{\mu < \la} W^q(\la)_{\mu}$. Thus the sum of two
proper submodules is again a proper submodule of $W^q(\la)$. Then
the sum $N^q(\la)$ of all proper submodules of $W^q(\la)$ is the
unique maximal submodule of $W^q(\la)$.
   \end{proof}

\noindent For $\la \in P$, ~ the unique irreducible quotient
$V^q(\la) := W^q(\la)/N^q(\la)$ is called the {\it irreducible
highest weight module} over $U_q (\gg)$ with highest weight $\la$
(defined up to $\Pi$).

  We introduce the notation
  \begin{equation*}
[n]_q := \dfrac{q^n - q^{-n}}{q - q^{-1}},
\end{equation*} which is called a {\it{$q$-integer}}.
We also define $[0]_{q}! := 1$ and $[n]_q ! := [n]_q \cdot [n-1]_q
\cdots [1]_q.$ We define the {\it divided powers} of $e_i$ and
$f_i$ as follows: \beq \bea e_i^{(k)} := \dfrac{e_i^k}{[k]_{q}!}, \
\ \ f_i^{(k)} := \dfrac{f_i^k}{[k]_{q}!}. \eea \eeq

By a straightforward induction argument, we can prove the following lemma.

\begin{lemma} \label{lemma_defining_rel_Vq(la)}For all $i \in I$ and $k \in \Z_{\geq 0}$, we have
    \begin{equation*}
e_i f_i^{(k)} = f_i^{(k)} e_i +
    f_i^{(k-1)} \dfrac{q^{h_i} q^{-k+1} - q^{-h_i} q^{k-1}}{q - q^{-1}}.
\end{equation*}
\end{lemma}

\begin{proposition} \label{defining_rel_Vq(la)}
  Let $\la \in \Lambda^{+}$ and $V^q(\la)$ be the irreducible highest weight
  $U_q(\gg)$-module generated by an irreducible finite dimensional
  $U_q^{\geq 0}$-module ${\bold v}^q$. Then $f_i^{\la(h_i)+1}v=0$ for all $v \in {\bold v}^q$ and $i \in I$.
\end{proposition}

\begin{proof}
Lemma \ref{lemma_defining_rel_Vq(la)} implies
\begin{equation*} \bea
e_i f_i^{(k)} v & = [{\la(h_i)-k+1}]_{q} f_i^{(k-1)} v \ \ \
\text{for all } v \in {\bold v}^q. \eea \end{equation*} If
$k=\la(h_i)+1$, we see that $e_i f_i^{\la(h_i)+1} v=0$. Moreover,
for $j \neq i$, we already know $e_j f_i^{\la(h_i)+1} v=0$ and
$e_{\bar j} f_i^{\la(h_i)+1} v=0$, since $V^q(\la) =
\bigoplus_{\mu \leq \la} V_{\mu}^q$.

Suppose that $e_{\bar i} f_i^{\la(h_i)+1} v \neq 0.$ We have \beq \bea
e_i(e_{\bar i} f_i^{\la(h_i)+1} v)&=e_{\bar i} (e_if_i^{\la(h_i)+1} v)=0,\\
e_{\bar i} (e_{\bar i} f_i^{\la(h_i)+1} v)&=-\dfrac{q-q^{-1}}{q+q^{-1}} e_{i}^2 f_i^{\la(h_i)+1} v =0.\\
\eea   \eeq
Also, $e_j(e_{\bar i} f_i^{\la(h_i)+1} v)=e_{\bar
j}(e_{\bar i} f_i^{\la(h_i)+1} v)=0$ for $j \neq i$, since $V^q(\la) =
\bigoplus_{\mu \leq \la} V_{\mu}^q$ .

If $\la(h_i) \geq 1$, then $\wt (e_{\bar i} f_i^{\la(h_i)+1} v)
=\la-\la(h_i)\al_i < \la$. Thus $e_{\bar i} f_i^{\la(h_i)+1} v$
would generate a nontrivial proper submodule of $V^q(\la)$, which
contradicts the irreducibility of $V^q(\la)$.

If $\lambda(h_i)=0$, then we have $\lambda_i=\lambda_{i+1}=0$ so
that $k_{\bar i} v= k_{\overline {i+1}} v=0$ by Lemma \ref{lambda_zero}. From the defining relation of $U_q(\gg)$,
we know
$$e_{\bar i}f_{i} v =f_ie_{\bar i} v + (q^{k_{i+1}} k_{\bar i}-q^{k_i} k_{\overline{i+1}})v=0.$$
Therefore, in any
case, $e_{\bar i} f_i^{\la(h_i)+1} v=0$ for all $v \in
\bold{v}^q$.

Similarly, if $f_i^{\la(h_i)+1} v \neq 0$, it would generate a
nontrivial proper submodule  of $V^q(\la)$. Hence we conclude
$f_i^{\la(h_i)+1} v=0$ for all $v \in \bold{v}^q$.
\end{proof}

\vskip 5mm

%%%%%%%%%%%%%%%%%%%%%%%%%%%%%%%%%%%%%%%%%%%%%%%%%%%%%%%%%%%%%%%%%%%%%%%%%%%%%%%%%%%%%%%%%%%%%%%%%%%%%%%%%%%%%%%%

\section{Classical limits } \label{section_limit}

Let $\Ao := \{f/g \in \C(q) \ | \  f, g \in \C[q], g(1) \neq 0 \}.$
For an integer $n \in \Z$, we formally define \beq
[y;n]_x:=\dfrac{yx^n-y^{-1}x^{-n}}{x-x^{-1}}, \ \ \ \
(y;n)_x:=\dfrac{yx^n-1}{x-1}. \eeq For example, \beq
[q^h;0]_q=\dfrac{q^h-q^{-h}}{q-q^{-1}}, \ \ \ \
(q^h;0)_q=\dfrac{q^h-1}{q-1}. \eeq

\begin{definition}
We define the {\em $\Ao$-form} $U_{\Ao}$ of the quantum
superalgebrta $U_q(\gg)$ to be the $\Ao$-subalgebra of $U_q(\gg)$
with 1 generated by the elements $e_i,e_{\bar i},f_i, f_{\bar i},
q^h, k_{\bar l}$ and $(q^h;0)_q$ ($i \in I, l \in J, h \in
P^{\vee}$).
\end{definition}

 We denote by $U_{\Ao}^+$ (respectively, $U_{\Ao}^-$) the $\Ao$-subalgebra of $U_q(\gg)$ with 1 generated by $e_i,e_{\bar i}$ (respectively, $f_i,f_{\bar i}$) for $i \in I$, and by $U_{\Ao}^0$ the $\Ao$-subalgebra of $U_q(\gg)$ with 1 generated by $q^h,  k_{\bar l}$ and $(q^h;0)_q$ for $ l \in J, \ h \in P^{\vee}$.

 \begin{lemma} \hfill \label{lemma for the tri dec of U(A_1)}
   \bee[{\quad \rm(1)}]
   \item $(q^h;n)_q \in \UAo^0$ for all $n \in \Z$ and $h \in P^{\vee}$.
   \item $[q^h;0]_q \in \UAo^0$  for all $n \in \Z$ and $h \in P^{\vee}$.
   \eee
 \end{lemma}

 \begin{proof}
  Our assertions follow immediately from the following identities:
   \beq \bea
   (q^h;n)_q & = q^n(q^h;0)_q + \dfrac{q^n-1}{q-1}, \\
   [q^h;0]_q & =  q\dfrac{q-1}{q^2-1}(1+q^{-h})(q^h;0)_q.
   \eea \eeq   \end{proof}

 Note that
 \beq
 k_{\bar i}^2=[q^{2k_i};0]_{q^2}=q^2\dfrac{q^2-1}{q^4-1}(1+q^{-2k_i})\dfrac{1}{q+1}(q^{2k_i};0)_q.
 \eeq

 \begin{proposition} \label{tri decom of UAo}
   We have the triangular decomposition of the algebra $\UAo$.  Namely,
   \beq
   \UAo \cong \UAo^- \ot \UAo^0 \ot \UAo^+
   \eeq
as $\Ao$-modules.
 \end{proposition}

\begin{proof}
  Recall the canonical isomorphism
  $\Uq(\gg) \stackrel{\sim}{\longrightarrow} \Uq^- \ot \Uq^0 \ot \Uq^+$
  given by Theorem \ref{tri decom}. The following commutation
  relations hold:
  \beq \bea
  &e_i(q^h;0)_q =(q^h;-\al_i(h))_q e_i, \ \ \ \ e_{\bar i}(q^h;0)_q =(q^h;-\al_i(h))_q e_{\bar i},\\
  &(q^h;0)_q f_i=f_i(q^h;-\al_i(h))_q , \ \ \ \ (q^h;0)_q f_{\bar i}=f_{\bar i}(q^h;-\al_i(h))_q,\\
&e_if_i = f_i e_i + [q^{k_i -k_{i+1}};0]_q,\\
&e_{i+1}f_i=q^{-1}f_ie_{i+1}, \ \ e_i f_{i+1}=q f_{i+1}e_i, \ \ \
e_i f_j -f_j e_i=0 \; \mbox{ for } |i-j|>1,\\
&e_{\bar i}f_{\bar i}= - f_{\bar i} e_{\bar i} + [q^{k_i +k_{i+1}};0]_q +(q-q^{-1})k_{\overline i}k_{\overline{i+1}},\\
&e_{\overline{i+1}}f_{\bar i}=-q^{-1}f_{\bar i}e_{\overline{i+1}},
\ \ e_{\bar i} f_{\overline{i+1}}=-q f_{\overline{i+1}}e_{\bar i},
\ \ \ e_{\bar i} f_{\bar j} = -f_{\bar j} e_{\bar i}=0 \; \mbox{
for } |i-j|>1. \eea \eeq Together with Lemma \ref{lemma for the
tri dec of U(A_1)}, one can show that the image of the canonical
isomorphism lies inside $\UAo^- \ot \UAo^0 \ot \UAo^+$ when
restricted to $\UAo$. Its inverse map is given by multiplication.
Hence the two spaces are isomorphic as $\Ao$-modules.
\end{proof}

In what follows, $V^q$ is a highest weight module over
$\Uq(\gg)$ with highest weight $\la \in P$ generated by a finite
dimensional irreducible $\Uq^{\geq 0}$-submodule ${\bold v^q}$.
Then ${\bold v^q}$ is a finite dimensional irreducible
$\Cliffqla$-module. Since it is irreducible, it is generated by a
nonzero vector $v \in ({\bold v^q})_{\bar 0}$; i.e., ${\bold v^q}=\Cliffqla v$.
Note that \beq \dfrac{q^{2n}-q^{-2n}}{q^2
-q^{-2}}=q^{2n-2}+q^{2n-6}+ \cdots + q^{-2n+6}+q^{-2n+2} \in \Ao \
\ \text{ for } n \in \Z_{>0}. \eeq
We denote by
$\mbox{Cliff}_{\Ao}(\lambda)$ the $\Ao$-subalgebra of
$\Cliff_q(\la)$ generated by $\{ t_{\bar i} ~|~ i \in J \}$.

\begin{definition}
 Let $V^q$ be a highest weight $U_q(\gg)$-module generated by a
finite dimensional irreducible $\Uqgeq$-module ${\bold v^q}$ and
let $E^{\Ao}(\lambda)$ be the $\CliffAola$-submodule of ${\bold
v^q} \cong E^q(\lambda)$ generated by a nonzero element $v \in ({\bold v^q})_{\bar 0}$. The {\em $\Ao$-form} of  $V^q$ is
defined to be the $\UAo$-submodule $V_{\Ao}$ of $V^q$ generated by
$E^{\Ao}(\lambda)$.
\end{definition}
In what follows, $V^q$ will denote a a highest weight $U_q(\gg)$-module.
%Fix a nonzero homogeneous vector $v \in {\bold v^q}$ and define $W:=\CliffAola v $. By a similar manner in Remark \ref{Uqgeq0-module and Cliffq-module}, we can identify a finite dimensional irreducible $\CliffAola$-module with a finite dimensional irreducible $\UAo^{\geq 0}$-module.

\begin{proposition} \label{VAo_is_free_UAo-module}
  $\VAo =\UAo^- E^{\Ao}(\lambda)$.
\end{proposition}

\begin{proof}
In view of Proposition \ref{tri decom of UAo}, it suffices to show
that $\UAo^+ E^{\Ao}(\lambda)=E^{\Ao}(\lambda)$ and $\UAo^0
E^{\Ao}(\lambda)=E^{\Ao}(\lambda)$. The first assertion is clear
by the definition of highest weight modules. For the second
assertion, we observe that
  \beq \bea
  q^h w&=q^{\la(h)}w,\\
  (q^h;0)_q w &=\dfrac{q^{\la(h)}-1}{q-1}w \ \ \ \text{for all } w \in E^{\Ao}(\lambda) .
 \eea \eeq
  Hence we obtain $\VAo=\UAo E^{\Ao}(\lambda)=\UAo^- E^{\Ao}(\lambda)$.
\end{proof}

For each $\mu \in P$, let us denote by $(V_{\Ao})_{\mu}$
the space $V_{\Ao} \cap V_{\mu}^q$. The following assertion can be
proved using the same arguments as in \cite[Proposition 3.3.6]{HK2002}.

\begin{proposition} \label{weight space decom of VAo}
 $V_{\Ao}$ has the weight space decomposition $V_{\Ao}=\bigoplus_{\mu \leq \la} (V_{\Ao})_{\mu}$.
\end{proposition}

\begin{proposition} \label{dim of weight space}
For each $\mu \in P$, the weight space $(V_{\Ao})_{\mu}$ is a free
$\Ao$-module with $\rank_{\Ao}(V_{\Ao})_{\mu}=\dim_{\C(q)}
V^q_{\mu}$. In particular,
$\rank_{\Ao}E^{\Ao}(\lambda)=\dim_{\C(q)} E^q(\lambda)$.
\end{proposition}
\begin{proof}
Because $\Ao$ is a principal ideal domain, every finitely
generated torsion free module over $\Ao$ is free. Furthermore,
since $\C(q)$ is the field of quotients of the integral domain
$\Ao$, a finite subset of a $\C(q)$-vector space is linearly
independent over $\C(q)$ if and only if it is linearly independent
over $\Ao$. Thus it is enough to show that each $V^q_{\mu}$ has a
$\C(q)$-basis which is also contained in $(V_{\Ao})_{\mu}$. The highest weight
space ${\bold v^q}=E^q(\la)$ has a
linearly independent subset of $\{t_{\bar 1}^{\ep_1} t_{\bar
2}^{\ep_2}\cdots t_{\bar n}^{\ep_n} v ~|~ \ep_j=0 \ {\rm or} \ 1 \}$ which
generates $E^q(\la)$ over $\C(q)$, since $E^q(\la)=\Cliff_q(\la)
v$. By definition, this subset is contained in $E^{\Ao}(\la)$. For
$V^q_{\mu}$, it is easy to show that there is a basis of
$V^q_{\mu}$ whose elements are of the form $f_\zeta t_{\bar 1}^{\ep_1} t_{\bar
2}^{\ep_2}\cdots t_{\bar n}^{\ep_n} v$, where $f_{\zeta}$
are monomials in $f_i$ and $f_{\bar j}$. This basis is also contained in
$(V_{\Ao})_{\mu}$, which proves the proposition.
\end{proof}

\begin{corollary}
The map $\phi:\C(q) \ot_{\Ao}V_{\Ao} \ra V^q$ given by $f \ot v
\map fv$ {\rm(}$f \in \C(q), v \in V_{\Ao}${\rm)} is a
$\C(q)$-linear isomorphism.
\end{corollary}

Let $\Jo$ be the ideal of $\Ao$ generated by $q-1$. Then there is
a canonical isomorphism of fields \beq \Ao/\Jo
\stackrel{\sim}{\longrightarrow} \C \ \ \ \text{given by  }
f(q)+\Jo \map f(1). \eeq
Define the $\C$-linear vector spaces \beq
\bea
U_1 &= (\Ao/\Jo)\ot_{\Ao} U_{\Ao}, \\
V^1 &=(\Ao/\Jo)\ot_{\Ao} V_{\Ao}. \eea \eeq
Then $V^1$ is
naturally a $U_1$-module. Note that $U_1 \cong U_{\Ao} / \Jo \UAo$
and $V^1 \cong V_{\Ao}/\Jo \VAo$. We use the bar notation for the
images under these maps. The passage under these maps is referred
to as taking the {\it classical limit}.

Since $\VAo = \UAo E^{\Ao}(\la)$, we have: \beq V^1 \cong
V_{\Ao}/\Jo \VAo =\UAo E^{\Ao}(\la) / \Jo \UAo
E^{\Ao}(\la)=(\UAo/\Jo \UAo ) \cdot (E^{\Ao}(\la)/\Jo
E^{\Ao}(\la)). \eeq
Hence $V^1$ is generated by $E^{\Ao}(\la)/\Jo
E^{\Ao}(\la)$ over $U^1$. For each $\mu \in P$, denote by
$V_{\mu}^1$ the space $(\Ao/\Jo) \ot_{\Ao} (V_{\Ao})_{\mu} \cong
(V_{\Ao})_{\mu}/\Jo(V_{\Ao})_{\mu}$.
\begin{proposition} \hfill \label{dim of weight space 2}
  \bee[{\quad \rm(1)}]
  \item $V^1=\bigoplus_{\mu \leq \la} V_{\mu}^1$
  \item For each $\mu \in P$, $\dim_{\C}V_{\mu}^1=\rank_{\Ao}(\VAo)_{\mu}$.
  \eee
\end{proposition}

\begin{proof}
The first assertion follows from Proposition \ref{weight space
decom of VAo}. Using the same argument as in \cite[Lemma
3.4.1]{HK2002}, we can prove the second assertion.
\end{proof}

Let $\bar h \in U_1$ be the classical limit of $(q^h;0)_q \in
\UAo$. Using \cite[Lemma 3.4.3]{HK2002}, we have:
\begin{lemma} \hfill
\bee[{\quad \rm(1)}]
\item For all $h \in p^{\vee}$, we have $\overline{q^h}=1$.
\item For any $h, h' \in P^{\vee}$, $\overline{h+h'}=\overline h + \overline{h'}$.
\eee
\end{lemma}

\begin{theorem} \hfill \label{from UAo to Uo}
\bee[{\quad \rm(1)}]
\item The elements $\overline{\Ei}, \overline{\Eib}, \overline{\Fi}, \overline{\Fib}$,
{\rm(}$i \in I${\rm)}, $\overline{\klb}$ {\rm(}$l \in J${\rm)}
and $\overline{h}$ {\rm(}$h \in P^{\vee}${\rm)} satisfy the defining relations of $U(\gg)$.
Hence there exists a surjective $\C$-algebra homomorphism
$\psi:U(\gg) \ra U_1$ and the $U_1$-module $V^1$ has a $U(\gg)$-module structure.

\item For each $\mu \in P$ and $h \in P^{\vee}$,
the element $\overline h$ acts on $V^1_{\mu}$ as scalar
multiplication by $\mu(h)$. So $V_{\mu}^1$ is the $\mu$-weight
space of the $U(\gg)$-module $V^1$.

\item There is an isomorphism $\Cliff(\lambda) \overset{\sim} \longrightarrow
\Cliff_1(\lambda):=\Cliff_{\Ao}(\lambda)/\Jo \Cliff_{\Ao}(\lambda)$.

\item As a $U(\gg)$-module, $V^1$ is a highest weight module or
the sum of two highest weight modules with highest weight $\la \in P$
\eee
\end{theorem}

\begin{proof}
\bee[{\quad \rm(1)}]
\item  The first relation for $U(\gg)$ is trivial. Since
\beq \bea
\qhh e_i -e_i \qhh &= \Ei (q^h;\al_i(h))_q -\Ei \qhh \\
                   &= \dfrac{q^{\al_i(h)}-1}{q-1}e_i q^h,
\eea \eeq we obtain $[\oh, \overline{e_i}]=\al_i(h)
\overline{e_i}$ by letting $q \rightarrow 1$. Similarly, \beq
[\oh, \oEib]=\al_i(h)\oEib, \ \ [\oh,
\overline{f_i}]=-\al_i(h)\overline{f_i}, \ \ [\oh,
\oFib]=-\al_i(h)\oFib \ \ \text{and} \ \  [\oh,\oklb]=0. \eeq

We have \beq \Ei \Fi -\Fi \Ei= [q^{h_i};0]_q = \dfrac{q}{q+1}
(1+q^{-h_i}) (q^{h_i};0)_q . \eeq Taking the classical limit to both sides above leads to
 $\oEi \oFi -\oFi \oEi =\dfrac{1}{2} 2
\overline{h_i}=\overline{h_i}$.

Also \beq
 k_{\bar i}^2=[q^{2k_i};0]_{q^2}=q^2\dfrac{q^2-1}{q^4-1}(1+q^{-2k_i})\dfrac{1}{q+1}(q^{2k_i};0)_q.
 \eeq
When we take $q \rightarrow 1$, we obtain $\overline{k_{\bar i}}^2
= \overline{k_{i}}.$

Since we can obtain the following relations in $U(\gg)$ by the
Jacobi identity,
 \beq \bea
 \ [ e_{\bar{i}}, [e_i,e_j]] &=[[e_{\bar i}, e_i], e_j]+[e_i,[e_{\bar i}, e_j]] = [e_i,[e_{\bar i}, e_j]],  \ \ \ \text{for  } |i-j|=1, \eea \eeq
in order to prove the corresponding relations in $U_1$, it suffices to show that
$[\oEi,[\oEib,\oEj]]=0$. The latter relation can
be checked easily by letting $q \to 1$. The rest of the relations can be derived in a
similar manner.

Therefore, there exists a surjective algebra homomorphism
$\psi:U(\gg) \ra U_1$ defined by $e_i \map \oEi, e_{\bar i} \map
\oEib, f_i \map \oFi, f_{\bar i} \map \oFib, h \map \oh, k_{\bar
l} \map \oklb$ ($i \in I, l \in J$), which can be used to define a
$U(\gg)$-module structure on $V^1$.

\item For $v \in (V_{\Ao})_{\mu}$ and $h \in P^{\vee}$, we have
\beq \qhh v =\dfrac{q^{\mu(h)}-1}{q-1}v. \eeq Taking the classical
limit of both sides yields our assertion.

\item Note that
$\overline{t_{\bar i}} \overline{t_{\bar j}} + \overline{t_{\bar
j}} \overline{t_{\bar i}} =2\delta_{ij} \la_i$ in $\Cliff_{1} (\lambda)$ and $\Cliff(\lambda)$ is
the associative $\C$-algebra with ${\bold 1}$ generated by
$\{k_{\bar i}~|~ i \in J \}$ with defining relations
$k_{\bar i} k_{\bar j} + k_{\bar j} k_{\bar i}=2\delta_{ij} \la_i$. Thus we have a surjective $\C$-algebra
homomorphism $\Cliff(\la) \rightarrow \Cliff_1 (\la)$. Observe
that
    \beq\bea
    \dim_{\C} \Cliff_1(\la) &= \rank_{\Ao} \Cliff_{\Ao}(\la)\\
                            &= \dim_{\C(q)} \Cliff_q(\la)\\
                            &= \dim_{\C} \Cliff(\la).
    \eea\eeq
The first two equalities follow by using the same reasoning as in
Proposition \ref{dim of weight space 2} and Proposition \ref{dim
of weight space}, respectively. It is well known that the
dimension of the Clifford algebra associated with a symmetric
bilinear form on a vector space of dimension $k$ is $2^k$. This result
holds for any base field of characteristic different from $2$. Thus we proved the last equality.

\item $V^q$ is generated by a finite dimensional irreducible $U_q^{\geq 0}$-submodule ${\bold v^q} \cong E^q(\la)$ up to $\Pi$. By Corollary \ref{cliff_mod}
    $$\dim E^q(\la)=\begin{cases}
    2^k   &\text{if } | \lambda| = 2k  \text{ and } \overline{\Delta(\lambda)} = \bar{1},\\
    2^{k+1}    &\text{if } | \lambda| = 2k \text{ and } \overline{\Delta(\lambda)} \neq \bar{1},\\
    2^{k+1}    &\text{if } | \lambda| = 2k+1.
  \end{cases}$$
It is well known that the dimension of the $\Z_2$-graded irreducible
$\Cliff(\la)$-modules is $2^{\left[ \frac{|\lambda| - 1}{2}\right]}  | 2^{\left[ \frac{|\lambda| - 1}{2}\right]} $ (see, for example, \cite{ABS}). With this in mind we deduce that
$E^{\Ao}(\la)/\Jo E^{\Ao}(\la)$ is an irreducible
$\Cliff(\la)$-module when $|\la|=2k+1$ or $|\la|=2k$ and
$\overline{\Delta(\lambda)}=1$, and the direct sum of two
irreducible $\Cliff(\la)$-modules otherwise. Since $E^q(\la)$ is a
parity invariant module over $\Cliff_q(\la)$ for $| \la |=2k$ and
$\overline{\Delta(\lambda)} \neq \bar1$, $E^{\Ao}(\la)/\Jo
E^{\Ao}(\la)$ is a parity invariant
$\Cliff(\lambda)$-module as well. Hence $E^{\Ao}(\la)/\Jo E^{\Ao}(\la)
={\bold v}(\lambda) \op \Pi {\bold v}(\lambda)$ for some
irreducible $\Cliff(\la)$-module ${\bold v}(\lambda)$. By
definition, $V^1$ is a highest weight $U(\gg)$-module generated by
$E^{\Ao}(\la)/\Jo E^{\Ao}(\la)$ or the sum of two highest weight
modules generated by ${\bold v}(\la)$ and $\Pi {\bold v}(\la)$ for
some irreducible $\Cliff(\la)$-module ${\bold v}(\la)$. \eee
\end{proof}

By Propositions \ref{dim of weight space} and \ref{dim of weight
space 2} and Theorem \ref{from UAo to Uo}, we obtain the following
identity between the characters of a highest weight $U(\gg)$-module and a highest weight
$U_q(\gg)$-module.
\begin{proposition} \label{ch_Vo=ch_V^q}
  $\ch V^1=\ch V^q$.
\end{proposition}

\begin{corollary}
$V^{q}(\lambda)$ is finite dimensional if and only if $\lambda \in
\Lambda^{+}$.
\end{corollary}
\begin{proof}
Let $V^{q}=V^q(\lambda)$. If $\lambda \in \Lambda^+$, then we have
$f_i^{\la(h_i)+1}v=0$ for all $v \in V^q_\lambda$ by Proposition
\ref{defining_rel_Vq(la)}. Taking the classical limit, we have
$\bar{f}_i^{\la(h_i)+1}\bar{v}=0$ for all $\bar{v} \in
V^1_\lambda$. Because $V^1$ is a highest weight module or the sum
of two highest weight modules, it is finite dimensional by
Proposition \ref{f.d. h.w. g-mod}, and hence $V^q$ is finite
dimensional by Proposition \ref{ch_Vo=ch_V^q}. Conversly, assume
that $\lambda$ is not in $\Lambda^{+}$. Then $V^1$ has a submodule
which is a highest weight module and whose irreducible quotient is
isomorphic to an irreducible highest weight module with highest
weight $\lambda$.  It is not finite dimensional by (2) of
Proposition \ref{f.d. irred g-mod}. Again by Proposition \ref{ch_Vo=ch_V^q},
$V^q$ cannot be finite dimensional.
\end{proof}

\begin{theorem} \label{V1_is_Uqn-module}
  If $\la \in \Lambda^+ \cap P_{\geq 0}$ and $V^q$ is
  the irreducible highest weight $\Uqqn$-module $V^q(\la)$
  with highest weight $\la$, then $V^1$ is isomorphic to
\begin{enumerate}[{\quad \rm(1)}]
\item $V(\la)$ or $\Pi V(\la)$ if $| \lambda| = 2k$  and $\overline{\Delta(\lambda)} = \bar{1}$,

\item $V(\la) \op \Pi V(\la)$ if $| \lambda| = 2k$ and
$\overline{\Delta(\lambda)} \neq \bar{1}$ {\rm(}in particular, if $\lambda_1 >....> \lambda_{2k} >0${\rm)},

\item $V(\la) \cong \Pi V(\la)$ if $|\lambda| = 2k+1$.
\end{enumerate}
Hence, $\ch V^q(\la)= \begin{cases}
     \ch V(\la)  &\text{if } | \lambda| = 2k  \text{ and } \overline{\Delta(\lambda)} = \bar{1},\\
    2 \ch V(\la) &\text{if } | \lambda| = 2k \text{ and } \overline{\Delta(\lambda)} \neq \bar{1},\\
    \ch V(\la)   &\text{if } | \lambda| = 2k+1.
   \end{cases} $
\end{theorem}

\begin{proof}
  By Theorem \ref{from UAo to Uo} (4), $V^1$ is a highest weight module or the sum of two highest weight modules over $U(\gg)$ with highest weight $\la$.  By Proposition \ref{f.d. h.w. module in O is irred.}, we have
     $$V^1 \cong \begin{cases}
     V(\la) \  \text{  or  } \  \Pi V(\la)   &\text{if } | \lambda| = 2k  \text{ and } \overline{\Delta(\lambda)} = \bar{1},\\
    V(\la) \op \Pi V(\la)    &\text{if } | \lambda| = 2k \text{ and } \overline{\Delta(\lambda)} \neq \bar{1},\\
    V(\la) \cong \Pi V(\la)   &\text{if } | \lambda| = 2k+1.
   \end{cases} $$

   The second assertion follows from Proposition \ref{ch_Vo=ch_V^q}.
\end{proof}

\begin{remark}The main reason we restrict our attention in Theorem \ref{V1_is_Uqn-module} to the dominant set of weights $ \Lambda^+ \cap P_{\geq 0}$ is the statement of Proposition \ref{f.d. h.w. module in O is irred.}. We still believe that the theorem holds in more general setting and conjecture that it is true for any weight $\lambda \in  \Lambda^+$ for which the generic character formula (\ref{character_formula})  holds.
\end{remark}

\begin{corollary} \label{irred of fd hw over uqqn}
If $V^q$ is a finite dimensional highest weight module over
$\Uqqn$ with highest weight $\lambda \in \Lambda^+ \cap P_{\geq
0}$, then $V^q$ is isomorphic to $V^q(\lambda)$ up to $\Pi$.
\end{corollary}

\begin{proof}
Note that $V^1$ is a highest weight module or the sum of two
highest weight modules over $U(\gg)$ with highest weight $\la$ and
it is finite dimensional by Proposition \ref{ch_Vo=ch_V^q}. From
Proposition \ref{f.d. h.w. module in O is irred.}, we know that
$V^1$ is an irreducible module or the direct sum of two
irreducible modules. Thus we get $\ch V^q = \ch V^1 =\ch
V^q(\lambda)$ by Theorem \ref{V1_is_Uqn-module} and hence $V^q
\cong V^q(\lambda)$.
\end{proof}

Define the subalgebras $U_1^{\pm} := \Ao/\Jo \ot_{\Ao}\UAo^{\pm}$ and $U_1^0
:= \Ao/\Jo \ot_{\Ao} \UAo^0$ of $U_1$.
\begin{theorem} \label{U1_is_Uqn-module}
  The classical limit $U_1$ of $\Uqqn$ is isomorphic to the universal enveloping algebra $U(\gg)$.
\end{theorem}

\begin{proof}
By Theorem \ref{from UAo to Uo} (1), there exists a surjective
algebra homomorphism $\psi: U(\gg) \ra U_1$ defined by $e_i \map
\oEi, \ e_{\bar i} \map  \oEib, \  f_i \map \oFi, f_{\bar i} \map
\oFib, \ h \map \overline{h}, k_{\bar l} \ra \overline{\klb}$ for
$i \in I$, $h \in P^{\vee}$ and $l \in J$. From \eqref{eq:tri
decomp of U(q(n))},  $U(\gg) \cong U^- \ot U^0 \ot U^+$.

We first show that $U^0$ is isomorphic to $U_1^0$. Consider the
restriction $\psi_0$ of $\psi$ to $U^0$. Note that
$\Cliff_{\Ao}(\la)$ is a $U^0_{\Ao}$-module. Indeed, as in the
proof of Proposition \ref{VAo_is_free_UAo-module}, we know that
  \beq \bea
  q^h w&=q^{\la(h)}w,\\
  (q^h;0)_q w &=\dfrac{q^{\la(h)}-1}{q-1}w \ \ \ \text{for all } w \in \Cliff_{\Ao}(\lambda).
 \eea \eeq
In particular, the action of $k_{\bar i}$ is just the left
multiplication by $t_{\bar i}$.  Let $g \in \ker\psi_0$. By the
Poincar\'e-Birkhoff-Witt theorem, we can write
$g=\sum_{i=1}^{2n} g_i k_{{\overline \eta_i}}$, where
$k_{{\overline \eta_i}}=k_{\bar{1}}^{a_1} \cdots
k_{\bar{n}}^{a_n}$, $0 \leq a_j \leq 1$ for all $j \in J$ and each
$g_i$ is a polynomial in $k_1, \ldots, k_n$. For each $\la \in P$
 we have
\beq 0=\psi_0(g)\cdot\overline{{\bold
1}}=\sum_{i=1}^{2n} \la(g_i) \overline{t_{{\overline \eta_i }}} \in
\Cliff_1(\la), \eeq
where $\la(g_i)$ denotes the polynomial in
$\la_j$ corresponding to $g_i$. Since $\{\overline{t_{{\overline \eta_i}}
} \}$ is a linearly independent subset of $\Cliff_1(\la) \cong
\Cliff(\la)$, we have $\la(g_i)=0$ for all $i=1,\ldots,2n$. Since
we may take any integer value for $\la_j$, $g_i$ must be
zero for all $i=1,\ldots,2n$ and hence $g$ is identically zero.
Thus $\psi_0$ is injective.

Next we show that the restriction of $\psi_-$ of $\psi$ to $U^-$
is an isomorphism of $U^-$ onto $U_1^-$. Suppose $\ker \psi_- \neq
0$ and $u=\sum b_{\zeta}f_{\zeta} \in \ker \psi_-$, where
$b_{\zeta} \in \C $ and $f_{\zeta}$ are monomials in $f_i$ and
$f_{\bar i}$'s. Let $N$ be the maximal length of the monomials
$f_{\zeta}$ in the expression of $u$, and choose $\la \in
\Lambda^+ \cap P_{\geq 0}$ satisfying $\la(h_i) >N$ and $|\la|=2k$
and $\Delta(\la)=\bar 1$ or $|\la|=2k+1$ for all $i \in I$. By
Theorem \ref{V1_is_Uqn-module}, the classical limit $V^1$ of
$V^q(\la)$ is isomorphic to the irreducible $U(\gg)$-module
$V(\la)$ when $|\la|=2k$ and $\Delta(\la)=\bar 1$, or
$|\la|=2k+1$. Set $r=2^{[\frac{|\la|+1}{2}]}$. Consider the map
$\phi: \big( U^- \big)^{\oplus r} \ra V^1$, given by
$(x_1,\ldots,x_n) \map \sum_{1=1}^r \psi(x_i) \cdot v_i$ for a
basis $\{v_i ~|~ i=1,\ldots,r\}$ of $V_{\la}^1$. Then by
Proposition \ref{f.d. h.w. module in O is irred.} and Proposition
\ref{f.d. h.w. g-mod}, $\ker \phi$ is the left ideal of $\big(U^-
\big)^{\op r}$ generated by $(f_i^{\la(h_i)+1},0,\ldots,0)$,
$\ldots$, $(0,\ldots,0,f_i^{\la(h_i)+1})$ for $i \in I$. In
particular, $(u,0,\ldots,0)=(\sum b_{\zeta}f_{\zeta},0,\ldots,0)
\not \in \ker \phi$. That is $\psi_-(u) v_1 \neq 0$, which is a
contradiction. So $\ker \psi_-=0$ and $U^-$ is isomorphic to
$U_1^-$.

Similarly, we can show that $U^+ \cong U_1^+$. By the triangular
decomposition we have \beq U(\gg) \cong U^- \ot U^0 \ot U^+ \cong
U_1^- \ot U_1^0 \ot U_1^+ \cong U_1. \eeq
It can be checked easily
that this isomorphism is an algebra isomorphism.
\end{proof}

\begin{theorem}
Let $\la \in P$. If $V^q$ is the Weyl module $W^q(\la)$ over
$\Uqqn$ with highest weight $\la$, then its classical limit $V^1$
is isomorphic to
\begin{enumerate} [{\quad \rm(1)}]
 \item $W(\la)$ or $\Pi W(\la)$ if $| \lambda| = 2k$  and $\overline{\Delta(\lambda)} = \bar{1}$,
 \item $W(\la) \op \Pi W(\la)$ if $| \lambda| = 2k$ and
 $\overline{\Delta(\lambda)} \neq \bar{1}$ {\rm(}in particular, if $\lambda_1 >...> \lambda_{2k} >0${\rm)},
 \item $W(\la) \cong \Pi W(\la)$ if $|\lambda| = 2k+1$.
\end{enumerate}
\end{theorem}

\begin{proof}
Let ${\bold v}(\la)$ be a finite dimensional irreducible
$\gb_+$-module of weight $\la$ which generates $W(\la)$. Since
$U^- \cong U_1^-$ and $E^{\Ao}(\la)/\Jo E^{\Ao}(\la)$ is
isomorphic to ${\bold v}(\la)$ or ${\bold v}(\la) \op \Pi {\bold
v}(\la)$  as a $\Cliff(\la)$-module, it suffices to show that
$V^1$ is a free $U_1^-$-module whose rank is $\dim_{\C}
{\bold v}(\la)$ or $2\dim_{\C} {\bold v}(\la)$.

By Proposition \ref{verma} we know that $W^q(\la)$ is a free $U_q^-$-module generated
by $E^q(\la)$. Since $\VAo$ is a subspace of $V^q$, taking
Proposition \ref{dim of weight space} into account, $\VAo$ is a
free $\UAo^-$-module generated by $E^{\Ao}(\la)$. Taking the
classical limit, we see that $V^1= U_1^- \cdot
\Big(E^{\Ao}(\la)/\Jo E^{\Ao}(\la) \Big)$ and
  \beq
  \dim_{\C} E^{\Ao}(\la)/\Jo E^{\Ao}(\la)=\dim_{\C(q)} E^q(\la)=\dim_{\C} {\bold v}(\la) \ \text{or} \ 2\dim_{\C} {\bold v}(\la).
  \eeq

By a similar argument as in \cite[Proposition 3.4.10]{HK2002}, we
can show that $V^1$ is a free $U_1^-$-module. When $| \lambda| =
2k$ and  $\overline{\Delta(\lambda)} \neq \bar{1}$, $E^q(\la)$ is
parity invariant. Hence we have
\begin{equation*}
V^1 \cong \begin{cases}
     W(\la) \  \text{  or  } \  \Pi W(\la)   &\text{if } | \lambda| = 2k  \text{ and } \overline{\Delta(\lambda)} = \bar{1},\\
    W(\la) \op \Pi W(\la)    &\text{if } | \lambda| = 2k \text{ and } \overline{\Delta(\lambda)} \neq \bar{1},\\
    W(\la) \cong \Pi W(\la)   &\text{if } | \lambda| = 2k+1.
   \end{cases}
\end{equation*}
\end{proof}

\vskip 5mm

%%%%%%%%%%%%%%%%%%%%%%%%%%%%%%%%%%%%%%%%%%%%%%%%%%%%%%%%%%%%%%%%%%%%%%%%%%%%%%%%%%%%%%%%%%%%%%%%%%%%%%%%%%%%%%%%%%%%%%

\section{Complete reducibility of the category ${\mathcal O}_{q}^{\geq 0}$} \label{section_semisimple}

In this section, we prove the complete reducibility theorem for
$U_q(\gg)$-modules in the category ${\mathcal O}_{q}^{\geq 0}$.

\begin{definition} \label{def_o_q}
The {\it category ${\mathcal O}_{q}^{\geq 0}$} consists of finite
dimensional $U_q(\gg)$-modules $M$ with a weight space
decomposition $M=\bigoplus_{\mu \in P} M_{\mu }$
 satisfying $(i)$ $\wt(M)
\subset P_{\geq 0} $,
$(ii)$ $k_{\bar i}|_{M_{\mu}} = 0 $ for $\mu \in P_{\geq 0}$ and $i \in \{1, \ldots, n\}$
such that $\langle k_i, \mu \rangle =0$.
\end{definition}

\begin{remark} The complete reducibility theorem for
${\mathcal O}_{q}^{\geq 0}$, which we establish at the end of this
section,  implies that  ${\mathcal O}_{q}^{\geq 0}$ is isomorphic to
the category ${\mathcal T}_q$ of tensor modules;  i.e., submodules
of a tensor power of the natural representation $\C(q)^{n | n}$.
Indeed, using the description of ${\mathcal T}_q$ provided by
Olshanski and Sergeev we first check that every simple object of
${\mathcal O}_{q}^{\geq 0}$ is a tensor module. Then, by the
complete reducibility result for ${\mathcal T}_q$, obtained again
by Sergeev and Olshanski, we conclude that the two categories are
isomorphic.
\end{remark}

\begin{proposition}
 (i)  The category ${\mathcal O}_{q}^{\geq 0}$ is closed under taking tensor products.

(ii) For each $\lambda \in \qdiw \cap P_{\geq 0}$, $V^q(\lambda)$ is an irreducible  $U_q(\gg)$-module  in the
category ${\mathcal O}_{q}^{\geq 0}$. Conversely,
every finite dimensional irreducible $U_q(\gg)$-module in the
category ${\mathcal O}_{q}^{\geq 0}$ has the form $V^q(\lambda)$
for some $\lambda \in \qdiw \cap P_{\geq 0}$.
 \end{proposition}
 \begin{proof}
For $(i)$ we use the comultiplication formula for $k_{\bar i}$. For $(ii)$ we use Theorem \ref{V1_is_Uqn-module} and Proposition \ref{prop_o}.
% One can easily prove the following proposition (see, for example, \cite[Theorem 7.2.3]{HK2002}).
\end{proof}

Let $S$ be the {\it antipode} on $U_q(\gg)$ defined in
\cite[Section 4]{O}. We have $S(q^h)=q^{-h}$ for all $h \in
P^{\vee}$. Because $S$ is an anti-automorphism on $U_q(\gg)$, one
can define two $U_q(\gg)$-module structures on the dual vector
space of a $U_q(\gg)$-module $V \in {\mathcal O}_q^{\geq 0}$ by
\begin{align*}
\langle x \cdot \phi , v \rangle &:= \langle \phi, S(x) \cdot v \rangle \ \rm{and} \\
\langle x \cdot \phi , v \rangle &:= \langle \phi, S^{-1}(x) \cdot
v \rangle
\end{align*}
for each $x \in U_q(\gg)$ and linear functional $\phi$ on $V$. We
denote these modules by $V^*$ and $V'$, respectively. As vector spaces both modules are
just $\bigoplus_{\mu \in P} V^*_{\mu}$, where
$V^*_{\mu} = \Hom_{\C (q)} (V_\mu , \C (q))$. The  following lemma
is an immediate consequence of the definitions.
\begin{lemma} \label{weight spaces of dual modules} Suppose that $V$ is a $U_q(\gg)$-module in the category ${\mathcal O}_q^{\geq 0}$.
\begin{enumerate}
\item[{\rm(1)}] There exist canonical $U_q(\gg)$-module isomorphisms $(V^*)' \cong V \cong (V')^* $.
\item[{\rm(2)}] The space $V^*_{\mu}$ is a weight space of weight $-\mu$.
\end{enumerate}
\end{lemma}

Since $q^h S(e_i)q^{-h}=q^{\al_i(h)}S(e_i)$, we have $S(e_i)
V_{\mu} \subset V_{\mu+\al_i}$, which implies $e_i V^*_{\mu}
\subset V^*_{\mu-\al_{i}}$. By Lemma \ref{weight spaces of dual
modules}, we get $e_i (V^*)_{-\mu} \subset (V^*)_{-\mu+\al_{i}}$.
Similarly, we also have $e_{\bar i} (V^*)_{-\mu} \subset
(V^*)_{-\mu+\al_{i}}$, $f_i (V^*)_{-\mu} \subset
(V^*)_{-\mu-\al_{i}}$, $f_{\bar i} (V^*)_{-\mu} \subset
(V^*)_{-\mu-\al_{i}}$ for all $i \in I$ and $k_{{\bar i}}
(V^*)_{-\mu} \subset (V^*)_{-\mu}$ for all $i \in J$. A weight
module $M$ is called a {\it lowest weight module} with lowest
weight $\lambda \in P$ if it is generated over $U_q (\gg)$ by an
irreducible finite dimensional $U_q^{\leq 0}$-module. By a similar
argument as in Proposition \ref{\Uq^0-module}, one can show that
$\big( V^q(\la)_{\la} \big)^*$ is an irreducible $U_q^{\leq
0}$-module so that $V^q(\lambda)^*$ and $V^q(\lambda)'$ are lowest
weight modules of lowest weight $-\lambda$.

Suppose that $V$ is a $U_q(\gg)$-module in the category ${\mathcal
O}_{q}^{\geq 0}$. Because $V$ is finite dimensional, we may choose
a {\it maximal} weight $\lambda \in \wt(V)$ with the property that
$\lambda + \alpha_i$ is not a weight of $V$ for any $i \in I$.
Then the weight space $V_{\lambda}$ is a completely reducible $U_q^{0}$-module (see the remark after Definition \ref{def_o}).
Fix an irreducible summand $\textbf{v}$ of
$V_{\lambda}$ and set $L=U_q(\gg)\bold{v}$. Then $L$ is a highest
weight $U_q(\gg)$-module with highest weight $\lambda$. By the
assumption, $\lambda \in \Lambda^+ \cap P_{\geq 0}$ and  from
Corollary \ref{irred of fd hw over uqqn} we know $L \cong
V^q(\lambda)$ up to $\Pi$.

Now consider $\bar{\bold v}=\Hom_{\C(q)}(\bold v, \C(q))$
Then it is a $U_q^{0}$-submodule of $(V^*)_\la$.
Set
$$\bar L = U_q(\gg)\bar{\bold v} \subset V^*.$$

It is easy to show that $\bar {\bold v}$ is an irreducible
$U_q^{\leq 0}(\gg)$-module and $\bar L$ is a lowest weight module
with lowest weight $-\lambda$. Translating Corollary \ref{irred of
fd hw over uqqn} to the case of lowest weight modules, we get
the following lemma.

\begin{lemma}
The $U_q(\gg)$-module $\bar L$ is isomorphic to the irreducible
lowest weight module $V^q(\lambda)^*$ with lowest weight
$-\lambda$ and lowest weight space ${\bar {\bold v}}$.
\end{lemma}

Now we can prove the completely reducibility theorem for
$U_q(\gg)$-modules in the category ${\mathcal O}_{q}^{\geq 0}$.

\begin{theorem}
Every $U_q(\gg)$-module $V$ in the category ${\mathcal
O}_{q}^{\geq 0}$ is completely reducible.
\end{theorem}

\begin{proof}
Take a maximal weight $\lambda$ and consider a submodule of $V$,
say $L$, generated by an irreducible $U_q^{\geq 0}$-submodule of
$V_{\lambda}$. We want to show $V \cong L \oplus V/L$. Taking dual
with respect to $S^{-1}$ of the inclusion $\bar L \to V^*$, we
obtain a $U_q(\gg)$-module homomorphism $ V \cong (V^*)' \to (\bar
L)'$. Thus we have a map:
$$\psi : L \hookrightarrow V \to (\bar L)'.$$
It is easy to check that $\psi$ is a nontrivial homomorphism.
Since both $L$ and $(\bar L)'$ are irreducible, $\psi$ is an
isomorphism by Schur's lemma and we see that the following short
exact sequence splits:
$$0 \to L \to V \to V/L \to 0.$$
Since $V/L \in {\mathcal O}_q^{\geq 0}$,  using induction on the
dimension of $V$, we complete the proof.
\end{proof}

\begin{corollary}
The tensor product of a finite number of $U_q(\gg)$-modules in the
category ${\mathcal O}_q^{\geq 0}$ is completely reducible.
\end{corollary}

\begin{remark}
The same argument can be applied to prove the completely reducibility of
${\mathcal O}^{\geq 0}.$ In that case, the antipode is given by
$S(x)=-x$ for all $x \in \gg$ (see \cite[Section 4]{N}) and
Proposition \ref{f.d. h.w. module in O is irred.} plays the same
role as Proposition \ref{irred of fd hw over uqqn}.
\end{remark}

\begin{acknowledgement}
  We would like to thank Ivan Penkov and Vera Serganova for the stimulating discussions. D.G. gratefully acknowledges the hospitality and excellent working conditions at the Seoul National University where most of this work was completed.
\end{acknowledgement}

%%%%%%%%%%%%%%%%%%%%%%%%%%%%%%%%%%%%%%%%%%%%%%%%%%%%%%%%%%%%%%%%%%%%%%%%%%%%%%%%%%%%%%%%%%%%%%%%%%%%%%%%%%%%%%%%%%%%%

\end{document}